\documentclass[11pt,leqno]{article}     
\usepackage[utf8]{inputenc}
\usepackage{amssymb,amsmath,latexsym,bbm}
\usepackage[english]{babel}
\usepackage{hyperref}
\usepackage{parskip}
\usepackage{enumitem}
\numberwithin{equation}{section}  
\usepackage{amsthm}     
\theoremstyle{plain}
\newtheorem{theorem}{Theorem}
\newtheorem*{theorem-non}{Theorem}
\newtheorem{corollary}[theorem]{Corollary}
\newtheorem{lemma}[theorem]{Lemma}
\newtheorem{proposition}[theorem]{Proposition}
\newtheorem*{proposition-non}{Proposition}
\numberwithin{theorem}{section}
\theoremstyle{definition}

\theoremstyle{remark}
\newtheorem{remark}[theorem]{Remark}
\usepackage{tikz-cd}

\begin{document}
\title{A note on the spaces of Eisenstein series on general congruence subgroups}
\author{Soumyadip Sahu\\School of Mathematics, Tata Institute of Fundamental Research,\\ 1 Homi Bhabha Road, Mumbai, 400005, Maharashtra, India.\\Email: soumyadip.sahu00@gmail.com} 
\date{}
\maketitle
\begin{abstract}
This article proposes a new approach to studying the spectral Eisenstein series of weight $k$ on a congruence subgroup of $\text{SL}_2(\mathbb{Z})$ using Hecke's theory of Eisenstein series for the principal congruence subgroups. Our method provides a gateway to analytic and arithmetic properties of the spectral Eisenstein series using corresponding results for the principal congruence subgroup. We show that the specializations of the weight $k$ spectral Eisenstein series at $s = 0$ give rise to a basis for the space of Eisenstein series on a general congruence subgroup, and the Fourier coefficients of the basis elements lie in a cyclotomic number field. Our philosophy also yields an explicit basis parameterized by cusps for the space of Eisenstein series with a nebentypus character. We utilize the spectral basis for the space of Eisenstein series to provide a simple proof of the Eichler-Shimura isomorphism theorem for the entire space of modular forms. 
\end{abstract}
  


\section{Introduction}
\label{section1}
Let $k$ be an integer $\geq 2$ and $\mathcal{G}$ be a congruence subgroup of $\text{SL}_2(\mathbb{Z})$. Suppose that $\mathcal{M}_{k}(\mathcal{G})$, resp. $\mathcal{S}_{k}(\mathcal{G})$, is the $\mathbb{C}$-vector space of modular forms, resp. cusp forms, on $\mathcal{G}$. We define the space of Eisenstein series on $\mathcal{G}$, denoted $\mathcal{E}_{k}(\mathcal{G})$, to be the linear complement of $\mathcal{S}_k(\mathcal{G})$ that is orthogonal to $\mathcal{S}_k(\mathcal{G})$ for the Petersson inner product. The nondegeneracy of the Petersson inner product on the space of cusp forms ensures that, if such a space exists, it is unique. Moreover, the dimension formulas for the spaces of modular forms and cusp forms predict the dimension of $\mathcal{E}_{k}(\mathcal{G})$. One uses Hecke's theory of Eisenstein series on principal congruence subgroups to construct the space of Eisenstein series on $\Gamma(N)$. Let $\mathcal{G}$ be a congruence subgroup that contains $\Gamma(N)$. Then the space of Eisenstein series on $\mathcal{G}$ equals  
\begin{equation}
\label{1.1}
\mathcal{E}_{k}(\mathcal{G}) = \mathcal{E}_k\big(\Gamma(N)\big) \cap \mathcal{M}_{k}(\mathcal{G}) = \mathcal{E}_k\big(\Gamma(N)\big)^{\mathcal{G}};
\end{equation}
see \cite[p.212]{ds}. In the spectral theory of the upper half plane \cite{iwaniec},  one attaches a collection of Eisenstein series of weight $0$ with the cusps of a congruence subgroup $\mathcal{G}$ of $\text{PSL}_2(\mathbb{Z})$ given by the formula 
\begin{equation}
\label{1.2}
\sum_{\gamma \in \mathcal{G}_x \backslash \mathcal{G}} (\text{Im } \sigma_{x}^{-1} \gamma \tau)^s
\end{equation}
where $x$ is a cusp of $\mathcal{G}$ and $\sigma_{x} \in \text{PSL}_2(\mathbb{Z})$ is a scaling matrix for the cusp $x$, i.e., $x = \sigma_{x} \infty$. However, most of the sources in the literature do not seem to treat the spectral Eisenstein of weight $k$ beyond special classes of congruence subgroups such as $\Gamma_{1}(N)$, $\Gamma_{0}(N)$, and $\Gamma(N)$; cf. \cite{kudla-yang}. This article stems from the author's effort to explicitly write down the spectral Eisenstein of weight $k$ on a general congruence subgroup with a view towards its arithmetic applications. Hecke's theory of Eisenstein series on principal congruence subgroups (Section~\ref{section2}) already provides a good construction of spectral Eisenstein series on $\Gamma(N)$. Our primary achievement is to express the spectral Eisenstein series of weight $k$ on a congruence subgroup $\mathcal{G}$ as a sum of Eisenstein series on a principal level $\Gamma(N)$ contained in $\mathcal{G}$ in a precise manner (Section~\ref{section4.1}). This task resembles summing over $\mathcal{G}$-orbits in $\mathcal{E}_k\big(\Gamma(N)\big)$ in the spirit of \eqref{1.1}, which, in terms, is closely related to the study of the $\mathcal{G}$ orbits in 
\[\{\bar{\lambda} \in \mathbb{Z}^{2}/N\mathbb{Z}^{2} \mid \text{ord}(\bar{\lambda}) = N\}\] for the natural right action of $\text{SL}_2(\mathbb{Z})$ on the lattice $\mathbb{Z}^{2}$. For practical purposes, one can write down these orbits for large classes of congruence subgroups, yielding a more concrete description for the spectral Eisenstein series; see Section~\ref{section3} and Appendix~\ref{appendix}. Our idea to break up the spectral Eisenstein series as a sum provides easy access to the analytic properties of the Eisenstein series on a congruence subgroup using the corresponding results for $\Gamma(N)$. In particular, we show that if $k \geq 3$ then the specializations of the spectral Eisenstein series of weight $k$ at $s = 0$ provide a basis for the space of Eisenstein series $\mathcal{E}_k(\mathcal{G})$. For $k = 2$, the specialization of the spectral Eisenstein series at $s = 0$ is not holomorphic due to an extra term arising from analytic continuation. Nevertheless, we provide explicit linear combinations of the specializations that form a basis of the holomorphic space $\mathcal{E}_2(\mathcal{G})$ (Section~\ref{section4.2}). Counting the number of basis elements, one arrives at the following result:     

\begin{proposition-non}\emph{(Corollary~\ref{corollary4.6new})}  
Let $\mathcal{C}(\mathcal{G})$, resp. $\mathcal{C}_{\infty}(\mathcal{G})$, denote the collection of cusps, resp. regular cusps, of $\mathcal{G}$. Then  
\[\emph{dim}_{\mathbb{C}} \mathcal{E}_k(\mathcal{G}) = \begin{cases}
			\lvert \mathcal{C}(\mathcal{G}) \rvert, & \text{if $k$ is even and $\geq 4$;}\\
			\lvert \mathcal{C}_{\infty}(\mathcal{G}) \rvert, & \text{if $k$ is odd and $-\emph{Id} \notin \mathcal{G}$;}\\
			\lvert \mathcal{C}(\mathcal{G}) \rvert -1, & \text{if $k = 2$;}\\
			0, & \text{otherwise.}
		\end{cases} \]
\end{proposition-non}

Thus, our work retrieves the dimension of the space of Eisenstein series on $\mathcal{G}$ assuming the standard description of the space of Eisenstein series on $\Gamma(N)$. If $k \geq 3$, then one does not need to know the expected dimension of $\mathcal{E}_{k}\big(\Gamma(N)\big)$ to verify that the construction using Hecke's Eisenstein series yields the correct subspace (Section~\ref{section2}). As a consequence, we obtain an elementary proof of the dimension count of the space of Eisenstein series of weight $\geq 3$ on a general congruence subgroup that avoids using heavy geometric tools such as the Riemann-Roch theorem. 

In standard literature, one finds two different avatars of the Eisenstein series on $\Gamma(N)$: the normalized series, which corresponds to the spectral Eisenstein series on $\Gamma(N)$, and the unnormalized series, which is convenient for practical computations since it possesses an explicit Fourier expansion (Section~\ref{section2}). The summation procedure mentioned above continues to hold if one replaces the normalized series on $\Gamma(N)$ with the unnormalized series and suggests a second construction of nonholomorphic Eisenstein series on congruence subgroups. We refer to this version of nonholomorphic Eisenstein series on a general congruence subgroup as the \textit{unnormalized} Eisenstein series (Section~\ref{section5}). Though the constant term of the unnormalized Eisenstein series lacks the familiar indicator of the cusp property, they do give rise to a basis of the space of Eisenstein series, and their explicit Fourier expansion makes this basis available for computations. The spectral and unnormalized Eisenstein series are related to each other through simple linear combinations whose coefficients are familiar Dirichlet series. Let $\mathcal{G}$ be a congruence subgroup containing $\Gamma(N)$. Suppose that $\mathcal{B}_{\mathcal{G},k}$ is the spectral basis of $\mathcal{E}_{k}(\mathcal{G})$ arising from the spectral Eisenstein series (Theorem~\ref{theorem4.5new}). Then the Fourier expansion formula for the unnormalized series yields the following striking rationality result:  

\begin{theorem-non}\emph{(Theorem~\ref{theorem5.5new})}
Each function in the spectral basis $\mathcal{B}_{\mathcal{G},k}$ admits a Fourier expansion of the form \[\sum_{j \geq 0} A_j\exp(\frac{2\pi i j\tau}{N})\] with $A_j \in \mathbb{Q}(\zeta_N)$ where $\zeta_N = \exp(\frac{2\pi i}{N})$.  
\end{theorem-non}

In the context of modular forms with a nebentypus character, it is conventional to write down a Hecke eigenbasis for the space of Eisenstein series instead of a basis parameterized by the cusps. We employ the summation philosophy discussed above to prescribe a basis parameterized by the cusps for the spaces of Eisenstein series with nebentypus (Section~\ref{section7.1new}), providing an alternate approach to a recent work of Young \cite{young} in this direction. Readers interested in the applications of such a basis may consult his article. The classical Eichler-Shimura theory for cusp forms defines a $\mathbb{C}$-linear Hecke equivariant homomorphism $\mathcal{S}_k(\mathcal{G}) \to H^{1}(\mathcal{G}, V_{k-2,\mathbb{C}})$ that formally extends to a $\mathbb{C}$-linear Hecke equivariant map 
\begin{equation*}
\mathcal{S}_k(\mathcal{G}) \oplus \mathcal{S}^{a}_k(\mathcal{G}) \to H^{1}(\mathcal{G}, V_{k-2,\mathbb{C}})  
\end{equation*} 
where $\mathcal{S}^{a}_k(\mathcal{G}) = \{\bar{f} \mid f \in \mathcal{S}_k(\mathcal{G})\}$ is the space of antiholomorphic cusp forms. The Eichler-Shimura isomorphism theorem for the cusp forms asserts that \eqref{1.1} is injective and its image equals the parabolic subspace of $H^{1}(\mathcal{G}, V_{k-2,\mathbb{C}})$ \cite[8.2]{shimura94}. There are several different ways in literature to extend the Eichler-Shimura isomorphism for cusp forms to a Hecke equivariant isomorphism: 
\begin{equation}
\label{1.3}
\mathcal{M}_k(\mathcal{G}) \oplus \mathcal{S}^{a}_k(\mathcal{G}) \cong H^{1}(\mathcal{G}, V_{k-2,\mathbb{C}}).
\end{equation}
The first proof of this result is using the theory of Eichler cohomology (see \cite{knopp} for details) that draws its tools from the deep theory of complex Riemann surfaces, while an algebro-geometric approach through the compactification of the modular curve appears in the work of Faltings \cite{faltings}. A relatively elementary approach based on the topology of Riemann surfaces is presented in the book of Hida \cite[6]{hida} that exploits the existence of a basis parameterized by the cusps to establish \eqref{1.3} in some special cases; cf. \cite[5]{bellaiche}. We use our spectral basis for the space of Eisenstein series to write down a complete proof of \eqref{1.3} following the topological strategy of Hida that works for  all weights $\geq 2$ and all congruence subgroups (Section~\ref{section6}). In principle, our proof is closest in spirit to the theory of Eisenstein cohomology of automorphic forms \cite{harder}, which realizes the same program in far greater generality using the Langlands Eisenstein series. The automorphic theory, however, does not state the isomorphism in terms of the finite level congruence subgroups, and deducing the results for congruence subgroups with torsion usually runs into technical difficulties. We also use our general isomorphism theorem to extend Hida's Eichler-Shimura isomorphism for modular forms with nebentypus to all levels and all nebentypus characters (Section~\ref{section7.2new}).

\subsection{Notation and conventions}
\label{section1.1}
Let $\mathbb{H}$ be the upper half-plane. The letter $\tau$ denotes a variable taking values in $\mathbb{H}$ and $\text{Im}(\tau)$ is the imaginary part of $\tau$. Suppose that $\text{Hol}(\mathbb{H})$, resp. $\mathcal{C}^{\omega}(\mathbb{H})$, is the algebra of $\mathbb{C}$-valued holomorphic, resp. real analytic, functions on $\mathbb{H}$. For $f \in \mathcal{C}^{\omega}(\mathbb{H})$ define the \textit{constant term} of $f$ to be \[\pi_{\infty}(f) := \lim_{t \to \infty} f(it)\] whenever the limit exists. We normalize the complex exponential function as $\mathbf{e}(z):= \exp(2 \pi i z)$. For a positive integer $N \geq 1$, set $\zeta_{N} := \textbf{e}(\frac{1}{N})$. 

Let $a_1, \ldots, a_{n}$ be integers. Then $\gcd(a_1, \ldots, a_{n})$ refers to the unique nonnegative generator of $\{a_1, \ldots, a_n\}\mathbb{Z}$. For example, $\gcd(0,0,2)=2$. Given a set $X$ let  $\mathbbm{1}_{X}$ be the delta function on $X$, i.e., 
\[\mathbbm{1}_{X}(x, y) := \begin{cases}
			1, & x = y;\\
			0, & x \neq y. 
		\end{cases}\]
		
Let $\Gamma$ be the full modular group $\text{SL}_2(\mathbb{Z})$. We deal with the standard congruence subgroups $\Gamma(N)$, $\Gamma_{1}(N)$, $\Gamma_{0}(N)$ \cite[1.2]{ds} and call $\Gamma(N)$ the \textit{principal congruence subgroup} of level $N$. Now suppose $\mathcal{G}$ is a general congruence subgroup. Then $\mathcal{C}(\mathcal{G})$, resp. $\mathcal{C}_{\infty}(\mathcal{G})$, refers to the collection of the cusps, resp. regular cusps, of $\mathcal{G}$. For $x \in \mathbb{P}^{1}(\mathbb{Q})$ let $\mathcal{G}_x$ denote the stabilizer of $x$ in $\mathcal{G}$. The stabilizer of $\infty$ in $\Gamma$ is $\Gamma_{\infty} = \{\pm T^n \mid n \in \mathbb{Z}\}$ where $T = \begin{pmatrix}  
			1 & 1\\
			0 & 1
		\end{pmatrix}$. 
For an integer $k$, we also need to consider the $\lvert_{k}$ action of $\Gamma$ on a function $f: \mathbb{H} \to \mathbb{C}$ defined as follows: 
\[f\lvert_{k, \gamma} := (c\tau + d)^{-k}f(\frac{a \tau + b}{c \tau + d})\] 
where $\gamma = \begin{pmatrix}
	a & b\\
	c & d
\end{pmatrix}$. If $k$ is clear from the context, we drop it from the notation. Note that the action preserves $\text{Hol}(\mathbb{H})$ and $\mathcal{C}^{\omega}(\mathbb{H})$.

\section{Eisenstein series on $\Gamma(N)$}
\label{section2}
This section revisits Hecke's theory of Eisenstein series on principal congruence subgroups \cite[4.2]{ds} with a view towards our applications. We begin by introducing a few notations used throughout the article. Let $\Lambda$ denote the lattice $\mathbb{Z}^{2}$. One writes the elements of $\Lambda$ as row vectors and treats it as a right $\Gamma$ module through matrix multiplication. For a positive integer $N$, set 
\[\Lambda_{N} = \{\bar{\lambda} \in \Lambda/N\Lambda \mid \text{ord}(\bar{\lambda}) = N\}.\] 
The $\Gamma$-structure on $\Lambda$ induces a $\Gamma$-action on $\Lambda_{N}$. Write $\Lambda_{N}^{\pm}= \Lambda_{N}/\{\pm \text{Id}\}$. It is clear that the $\Gamma$-action on $\Lambda_{N}$ descends to a well-defined action on $\Lambda_{N}^{\pm}$. The following result recalls the characterization of the cusps of $\Gamma(N)$.  

\begin{proposition}
\label{proposition2.1} \emph{(\cite[3.8]{ds}, Proposition~3.8.3)}  
Let $x = \frac{\alpha}{\beta}$ and $x'= \frac{\alpha'}{\beta'}$ be elements of $\mathbb{P}^{1}(\mathbb{Q})$ where $(\alpha, \beta), (\alpha', \beta') \in \Lambda$ with $\gcd(\alpha, \beta) = \gcd(\alpha', \beta') =1$. Then $x_N = x'_N$ if and only if $(\alpha, \beta) \equiv \pm (\alpha', \beta') (\emph{mod $N$})$. Here $(\cdot)_{N}$ refers to the image of $(\cdot)$ in $\mathcal{C}\big(\Gamma(N)\big)$. 
\end{proposition}
	
The $\Gamma$-action on $\mathbb{P}^{1}(\mathbb{Q})$ gives rise to a natural $\Gamma$-action on $\mathcal{C}\big(\Gamma(N)\big)$ described by $\gamma . \Gamma(N)x = \Gamma(N)(\gamma x)$ where $\gamma \in \Gamma$ and $x \in \mathbb{P}^{1}(\mathbb{Q})$.
One can rephrase Proposition~\ref{proposition2.1} as a $\Gamma$-equivariant bijection. Define 
\begin{equation}
\label{2.1new}
\phi_N: \mathcal{C}\big(\Gamma(N)\big) \to \Lambda_N^{\pm}, \hspace{.3cm} \phi_N(x) := [\overline{(\beta, -\alpha)}]
\end{equation}
where $x = (\frac{\alpha}{\beta})_N$ with $\gcd(\alpha, \beta) = 1$ and $[\cdot]$ denotes the image in $\Lambda_N^{\pm}$. Proposition~\ref{proposition2.1} demonstrates that $\phi_{N}$ is a well-defined injection. Moreover, a direct calculation verifies that $\phi_N$ is $\Gamma$-equivariant, i.e., $\phi_N(\gamma x) = \phi_N(x)\gamma^{-1}$ for each $\gamma \in \Gamma$. Note that $\phi_N(\infty_{N}) = [\overline{(0, 1)}]$. 
	
\begin{lemma}
\label{lemma2.2}
The map $\phi_N$ is a $\Gamma$-equivariant bijection. 
\end{lemma}    
\begin{proof}
We need to check that $\phi_N$ is onto. Let $\overline{(\alpha, \beta)} \in \Lambda_N$. It is clear that $\gcd(\alpha, \beta, N) = 1$. Therefore $\overline{(\alpha, \beta)}$ admits a lift $(\alpha', \beta') \in \Lambda$ with $\gcd(\alpha', \beta') = 1$. Set $x = (\frac{-\beta'}{\alpha'})_N$. Then $\phi_N(x) = [\overline{(\alpha, \beta)}]$.    
\end{proof}

Next, we introduce the Eisenstein series on $\Gamma(N)$ that play the central role in our treatment of the Eisenstein series on a general congruence subgroup. Let $k \geq 2$ and $\bar{\lambda} \in \Lambda_N$. Set \[\epsilon_{N} = \begin{cases}
	1, & \text{if $N \geq 3$;}\\
	\frac{1}{2}, & \text{if $N \in \{1,2\}$.}
\end{cases}\] 
For a complex variable $s$ one defines the \textit{normalized} Eisenstein of weight $k$ on $\Gamma(N)$ by  
\begin{equation}
\label{2.2new}
E_{k}(\tau, \bar{\lambda}, N; s) := \epsilon_{N} \sum_{\substack{(c,d) \equiv \lambda (N),\\ \gcd(c,d) = 1}} \frac{\text{Im}(\tau)^s}{(c\tau+d)^k \lvert c\tau+d \rvert^{2s}}
\end{equation}
where $\lambda$ is a lift of $\bar{\lambda}$ to $\Lambda$. We also need a \textit{unnormalized} variant of the series above described as  
\begin{equation}
\label{2.3new}
G_{k}(\tau, \bar{\lambda}, N; s) := \sum_{\substack{(c,d) \equiv \lambda (N),\\ (c,d) \neq (0,0)}} \frac{\text{Im}(\tau)^s}{(c\tau+d)^k \lvert c\tau+d \rvert^{2s}}.
\end{equation}
These two series converge absolutely and uniformly on compact subsets of $\{s \mid \text{Re}(k+2s) > 2\}$ to define analytic functions of $s$ in this region. Assume that $\text{Re}(k + 2s) > 2$. A simple rearrangement argument now yields the following well-known identities: 
\begin{equation}
\label{2.4new}
\begin{gathered}
G_{k}(\tau, \bar{\lambda}, N; s) = \epsilon_{N}^{-1} \sum_{\bar{\delta} \in (\mathbb{Z}/N\mathbb{Z})^{\times}} \zeta_{+}^{(N)}(\delta, k+2s) E_k(\tau, \bar{\delta}^{-1}\bar{\lambda}, N; s),\\
E_{k}(\tau, \bar{\lambda}, N; s) = \epsilon_{N} \sum_{\bar{\delta} \in (\mathbb{Z}/N\mathbb{Z})^{\times}} \zeta_{+}^{(N)}(\mu; \delta, k+2s) G_k(\tau, \bar{\delta}^{-1}\bar{\lambda}, N; s).  
\end{gathered}
\end{equation}   
Here $\delta$ is a lift of $\bar{\delta}$ to $\mathbb{Z}$, $\mu$ is the M\"obius function, and 
\[\zeta_{+}^{(N)}(\delta, k+2s) := \sum_{\substack{n \geq 1, \\ n \equiv \delta (N)}} \frac{1}{n^{k+2s}},\; \zeta_{+}^{(N)}(\mu; \delta, k+2s) := \sum_{\substack{n \geq 1, \\ n \equiv \delta (N)}} \frac{\mu(n)}{n^{k+2s}}.\] Moreover for each $\gamma \in \Gamma$ we have 
\begin{equation}
\label{2.5new}
\begin{gathered}
E_{k}(\tau, \bar{\lambda}, N; s)\lvert_{k, \gamma} = E_{k}(\tau, \bar{\lambda}\gamma, N; s),\\
G_{k}(\tau, \bar{\lambda}, N; s)\lvert_{k, \gamma} = G_{k}(\tau, \bar{\lambda}\gamma, N; s).     
\end{gathered}  
\end{equation}
The function $G_{k}(\tau, \bar{\lambda}, N; s)$ admits an analytic continuation to the whole $s$-plane \cite[9.1]{shimura}. One uses \eqref{2.4new} to analytically continue $E_{k}(\tau, \bar{\lambda},N; s)$ to the larger domain $\{s \mid \text{Re}(k+2s) > 1\}$. It is clear that \eqref{2.5new} holds in the corresponding regions of analytic continuation. We abbreviate the specialization of the normalized, resp. unnormalized, series at $s = 0$ by 
$E_{k}(\tau, \bar{\lambda}, N)$, resp. $G_{k}(\tau, \bar{\lambda}, N)$. If $k \geq 3$ then $E_{k}(\tau, \bar{\lambda}, N)$ and $G_{k}(\tau, \bar{\lambda}, N)$ are holomorphic functions on $\mathbb{H}$. In general, as a function on $\mathbb{H}$ the Fourier expansion of the $G_{k}(\tau, \bar{\lambda}, N)$ is  
\begin{equation}
\label{2.6new}
\begin{gathered}
G_{k}(\tau, \bar{\lambda}, N) = \sum_{j \geq 0} A_{k,j}(\bar{\lambda}) \mathbf{e}(\frac{j\tau}{N}) + P_{k}(\tau);\hspace{.3cm}\substack{(k \geq 2)}\\
A_{k, 0}(\bar{\lambda}) := \mathbbm{1}_{\mathbb{Z}/N\mathbb{Z}}(\bar{\lambda}_1, 0) \sum_{\substack{n \in \mathbb{Z} - \{0\}, \\ n \equiv \lambda_2(N)}} \frac{1}{n^k},\;P_k(\tau) := -\mathbbm{1}_{\mathbb{Z}}(k,2)\frac{\pi}{N^2\text{Im}(\tau)}\\
A_{k,j}(\bar{\lambda}) := \frac{(-2\pi i)^k}{(k-1)! N^k} \sum_{\substack{r \in \mathbb{Z} - \{0\},\\ r \mid j, \frac{j}{r}\equiv \lambda_1(N)}} \text{sgn}(r)r^{k-1}\mathbf{e}(\frac{r\lambda_2}{N}) \hspace{.4cm}(\substack{j \geq 1})
\end{gathered}
\end{equation}
where $(\lambda_1, \lambda_2)$ is a lift of $\bar{\lambda}$ to $\Lambda$ (\textit{loc. cit.}). It follows that $G_2(\tau, \bar{\lambda}, N) \in -\frac{\pi}{N^2\text{Im}(\tau)} + \text{Hol}(\mathbb{H})$. We employ \eqref{2.4new} to conclude  
\begin{equation}
\label{2.7new}
E_2(\tau, \bar{\lambda}, N) \in - \frac{6\epsilon_{N}}{\pi N^2 \prod_{p \mid N}(1 - \frac{1}{p^2})\text{Im}(\tau)} + \text{Hol}(\mathbb{H}) \subseteq \mathcal{C}^{\omega}(\mathbb{H}). 
\end{equation}
The Fourier expansion formula \eqref{2.6new} shows that the value of the $G$-series at $\infty$ is given by $\pi_{\infty}\big(G_{k}(\tau, \bar{\lambda}, N)\big) = A_{k, 0}(\bar{\lambda})$. Now, one uses \eqref{2.4new} and the M\"obius inversion identity to discover that  
\begin{equation}
\label{2.8new}
\begin{aligned}
&\pi_{\infty}\big(E_{k}(\tau, \bar{\lambda}, N)\big)\\  &=\epsilon_{N}\mathbbm{1}_{\mathbb{Z}/N\mathbb{Z}}(\bar{\lambda}_1,0)\big(\mathbbm{1}_{\mathbb{Z}/N\mathbb{Z}}(\bar{\lambda}_2, 1) + (-1)^k\mathbbm{1}_{\mathbb{Z}/N\mathbb{Z}}(\bar{\lambda}_2, -1)\big).    
\end{aligned}
\end{equation} 
If $N \geq 3$ then at most one of $\mathbbm{1}_{\mathbb{Z}/N\mathbb{Z}}(\bar{\lambda}_2, \pm 1)$ is nonzero. But, for $N = 1, 2$, both of $\mathbbm{1}_{\mathbb{Z}/N\mathbb{Z}}(\bar{\lambda}_2, \pm 1)$ can be nonzero. This phenomenon demystifies the presence of the factor $\epsilon_{N}$. 

\begin{remark}
\label{remark2.3}
Readers interested in the nonholomorphic aspect of the theory may consult the proof of Theorem~A3.5 in \cite[p.134]{shimura} for a full Fourier expansion of the unnormalized series given in terms of the Whittaker function. 
\end{remark}
				
Finally, we explain how to use the Eisenstein series introduced above to construct the space of Eisenstein series and choose a basis for this space parameterized by cusps. To provide uniform statements for all weights $\geq 2$, one first introduces the notion of an extended space of Eisenstein series that contains nonholomorphic functions if $k = 2$.

Let $k$ be an integer $\geq 2$ and $\mathbbm{k}$ be a subring of $\mathbb{C}$. We define the \textit{extended} space of Eisenstein series with coefficients in $\mathbbm{k}$, denoted $\mathcal{E}^{\ast}_{k}\big(\Gamma(N), \mathbbm{k}\big)$, as the $\mathbbm{k}$-submodule generated by $\{E_{k}(\bar{\lambda}, N) \mid \bar{\lambda} \in \Lambda_N\}$ inside the space of real analytic functions $\mathcal{C}^{\omega}(\mathbb{H})$. 
The transformation formula in \eqref{2.5new} implies that $\mathcal{E}^{\ast}_{k}\big(\Gamma(N), \mathbbm{k}\big)$ is stable under the $\lvert_{k}$-action of $\Gamma$. If $k$ is odd and $N \in \{1,2\}$ then $\mathcal{E}^{\ast}_{k}\big(\Gamma(N), \mathbbm{k}\big) = \{0\}$. Assume that either $k$ is even or $N \geq 3$. We have a well-defined $\mathbbm{k}$-linear map \[\pi_{\infty}: \mathcal{E}^{\ast}_{k}\big(\Gamma(N), \mathbbm{k}\big) \to \mathbb{C}; \hspace{.3cm}f \mapsto \lim_{t \to \infty} f(it).\] 
The identity \eqref{2.8new} shows that the function $\pi_{\infty}\big(E_{k}(\cdot, N)\big)$ acts like the indicator of the subset $\{\pm \overline{(0, 1)}\} \subseteq \Lambda_N$. Now suppose $\mathcal{A}$ is a set of representatives for the orbits of the $\{\pm \text{Id}\}$-action on $\Lambda_N$. Let $\Pi_N: \Lambda_N \to \Lambda_N^{\pm}$ be the canonical $\Gamma$-equivariant projection. Write $\Pi_{N, \mathcal{A}} := \Pi_N\bigl\lvert_{\mathcal{A}}$. Note that $\Pi_{N, \mathcal{A}}$ is a bijection. Define  
\begin{equation}
\label{2.9new}
\Phi_{N, \mathcal{A}}: \mathcal{C}\big(\Gamma(N)\big) \to \mathcal{E}^{\ast}_k\big(\Gamma(N), \mathbbm{k}\big); \hspace{.3cm} x \mapsto E_k\big(\Pi_{N, \mathcal{A}}^{-1}\circ \phi_N(x), N\big). 
\end{equation}

A familiar argument involving the indicator function property of $\pi_{\infty}$ and the transitivity of $\Gamma$-action on $\Lambda_{N}^{\pm}$ shows that $\{\Phi_{N, \mathcal{A}}(x) \mid x \in \mathcal{C}\big(\Gamma(N)\big)\}$ is a $\mathbbm{k}$-basis for $\mathcal{E}^{\ast}_k\big(\Gamma(N), \mathbbm{k}\big)$; see Theorem~4.2.3 in \cite[4.2]{ds}. It is clear that $\Phi_{N, \mathcal{A}}$ is independent of $\mathcal{A}$ if $k$ is even. Now suppose $\gamma \in \Gamma$. Since both $\phi_{N}$ and $\Pi_{N}$ are $\Gamma$-equivariant, it follows that $\Pi_{N, \mathcal{A}}^{-1} \circ \phi_{N}(\gamma x) = \pm \Pi_{N, \mathcal{A}}^{-1} \circ \phi_{N}(x)\gamma^{-1}$. As a consequence,  
\begin{equation}
\label{2.10new}
\Phi_{N, \mathcal{A}}(\gamma x) = \varepsilon_{\gamma}(x)^k \Phi_{N, \mathcal{A}}(x)\lvert_{\gamma^{-1}}. \hspace{.3cm}(\substack{\varepsilon_{\gamma}(x) \in \{\pm 1\}, \; x \in \mathcal{C}(\Gamma(N))})
\end{equation}
This reformulation of the classical description of basis plays a key role in our construction of the spectral basis on a general congruence subgroup, especially in the odd weight case. 

Let the notation be as above. One defines the (holomorphic) space of Eisenstein series on $\Gamma(N)$ with coefficients in $\mathbbm{k}$ to be \[\mathcal{E}_k\big(\Gamma(N), \mathbbm{k}\big) := \mathcal{E}^{\ast}_k\big(\Gamma(N), \mathbbm{k}\big) \cap \text{Hol}(\mathbb{H}).\] 
It is clear that $\mathcal{E}_k\big(\Gamma(N), \mathbbm{k}\big)$ is also stable under $\lvert_{k}$-action of $\Gamma$. If $k \geq 3$ then the spaces $\mathcal{E}^{\ast}_k\big(\Gamma(N), \mathbbm{k}\big)$ and $\mathcal{E}_k\big(\Gamma(N), \mathbbm{k}\big)$ coincide. A straight forward computation using \eqref{2.7new} shows that $\mathcal{E}_2\big(\Gamma(N), \mathbbm{k}\big)$ equals 
\begin{equation}
\label{2.11new}
\big\{\sum_{x \in \mathcal{C}(\Gamma(N))} C_x \Phi_{N, \mathcal{A}}(x) \mid \sum_{x} C_x = 0\big\}.	
\end{equation}
Moreover, with the notation of \eqref{2.9new}, the collection $\big\{\Phi_{N, \mathcal{A}}(x) - \Phi_{N, \mathcal{A}}(\infty_N) \mid x \neq \infty_N\big\}$ is a $\mathbbm{k}$-basis for the free $\mathbbm{k}$-module $\mathcal{E}_2\big(\Gamma(N), \mathbbm{k}\big)$.    
	
It remains to check that the space $\mathcal{E}_{k}\big(\Gamma(N), \mathbb{C}\big)$ satisfies the abstract criteria of the space of Eisenstein series provided at the beginning of the introduction. The explicit bases constructed using \eqref{2.9new} ensures that $\mathcal{E}_{k}\big(\Gamma(N), \mathbb{C}\big)$ is linearly disjoint from $\mathcal{S}_k\big(\Gamma(N)\big)$. A standard calculation in classical theory demonstrates that the $E$-series is orthogonal to the space of cusp forms \cite[5.11]{ds}. To complete the argument we need to verify that $\mathcal{E}_{k}\big(\Gamma(N),\mathbb{C}\big)$ and $\mathcal{S}_{k}\big(\Gamma(N)\big)$ together span $\mathcal{M}_k\big(\Gamma(N)\big)$. One way to achieve this result is to invoke dimension formulas for $\mathcal{M}_k\big(\Gamma(N)\big)$ and $\mathcal{S}_k\big(\Gamma(N)\big)$. However, if $k \geq 3$ then it is possible to write down a direct proof.    
	
\begin{proposition}
\label{proposition2.4}
Let $k \geq 3$. Then \[\mathcal{M}_k\big(\Gamma(N)\big) = \mathcal{E}_{k}\big(\Gamma(N), \mathbb{C}\big) + \mathcal{S}_k\big(\Gamma(N)\big).\] 	
\end{proposition}
	
\begin{proof}
If $k$ is odd and $N \in \{1,2\}$ then there is nothing to prove. We assume that $N \geq 3$ whenever $k$ is odd. Let $f \in \mathcal{M}_k\big(\Gamma(N)\big)$. Then $\pi_{\infty}(f\lvert_{\gamma})$ depends only on the image of $\gamma$ in $\Gamma(N) \backslash \Gamma / \Gamma^{+}_{\infty}$ where $\Gamma^{+}_{\infty} = \{T^n \mid n \in \mathbb{Z}\}$. Let $\{\gamma_1, \ldots, \gamma_n\}$ be a collection of representatives for the double cosets in $\Gamma(N) \backslash \Gamma / \Gamma_{\infty}$. Define $f_{\text{c}} = f - \sum_{j = 1}^{n} \pi_{\infty}(f\vert_{ \gamma_j})E_{k}\big(\overline{(0, 1)}, N\big)\lvert_{\gamma_{j}^{-1}}$. We need to show that $f_c \in \mathcal{S}_k\big(\Gamma(N)\big)$. Note that $\{\gamma \mid \overline{(0, 1)}\gamma = \pm \overline{(0, 1)}\} = \Gamma_{\infty}\Gamma(N)$ and $\gamma_{j}^{-1}\gamma_r \notin \Gamma_{\infty}\Gamma(N)$ whenever $j \neq r$. Therefore $\pi_{\infty}(f_{c}\lvert_{\gamma_{j}}) = 0$ for each $1 \leq j \leq n$. Since $\pi_{\infty}(f_{c}\lvert_{-\gamma}) = (-1)^k \pi_{\infty}(f_{c}\lvert_{\gamma})$ it follows that  $\pi_{\infty}(f_{c}\lvert_{\gamma}) = 0$ for each $\gamma \in \Gamma$. Hence $f_c \in \mathcal{S}_{k}\big(\Gamma(N)\big)$ as required.        
\end{proof}

\section{Group action on the mod $N$ lattice}
\label{section3} 
Let $\mathcal{G}$ be a congruence subgroup containing $\Gamma(N)$. A naive attempt to construct $\mathcal{G}$-invariant combinations inside $\mathcal{E}_k\big(\Gamma(N)\big)$ in the spirit of \eqref{1.1} leads to the study of $\mathcal{G}$-orbits inside $\mathcal{C}\big(\Gamma(N)\big)$. Let $\mathcal{C}(\mathcal{G}, N)$ denote the collection of $\mathcal{G}$-orbits in $\mathcal{C}\big(\Gamma(N)\big)$. There is a natural bijection  
\begin{equation*}
	\text{or}_{\mathcal{G},N}: \mathcal{C}(\mathcal{G}) \xrightarrow{\cong} \mathcal{C}\big(\mathcal{G}, N\big); \hspace{.3cm} \mathcal{G}x \mapsto \mathcal{G}\big(\Gamma(N)x\big)
\end{equation*}
that sends a cusp of $\mathcal{G}$ to the corresponding $\mathcal{G}$ orbit in $\mathcal{C}\big(\Gamma(N)\big)$. The map $\phi_{N}$ in Section~\ref{section2} induces a bijection between $\mathcal{C}(\mathcal{G}, N)$ and the collection of $\mathcal{G}$ orbits in $\Lambda_{N}^{\pm}$. However, the latter correspondence is not sufficient for our purpose of constructing a basis since here one needs to keep track of the sign ambiguity along the orbits of $\mathcal{G}$ in $\Lambda_{N}$.  

Let $\mathcal{O} \subseteq \Lambda_{N}$ be an orbit for the action of $\mathcal{G}$ on $\Lambda_{N}$. Then $-\mathcal{O}$ is also $\mathcal{G}$-orbit in $\Lambda_{N}$ and there are two mutually disjoint possibilities, namely, $\mathcal{O} \cap (-\mathcal{O}) = \emptyset$ or $\mathcal{O} = -\mathcal{O}$. We say a $\mathcal{G}$-orbit $\mathcal{O}$ in $\Lambda_N$ is \textit{regular}, resp. \textit{irregular}, if $\mathcal{O} \cap (-\mathcal{O}) = \emptyset$, resp. $\mathcal{O} = - \mathcal{O}$. Note that an orbit $\mathcal{O}$ is regular, resp. irregular, if and only if $-\mathcal{O}$ is also regular, resp. irregular. The following result relates the regularity of orbits to the regularity of cusps \cite[p.75]{ds}. Let $\mathcal{C}_{\infty}(\mathcal{G})$ be the collection of regular cusps of $\mathcal{G}$. Suppose that $\mathcal{C}_{\infty}(\mathcal{G}, N)$ is the image of $\mathcal{C}_{\infty}(\mathcal{G})$ under $\text{or}_{\mathcal{G},N}$.  

\begin{proposition}
\label{proposition3.1new}\emph{(cf. \cite[4]{cummins et al})} 
Assume $-\emph{Id} \notin \mathcal{G}$. Then, a $\mathcal{G}$-orbit $\mathcal{O}$ in $\Lambda_{N}$ is regular if and only if \[\phi_N^{-1}([\mathcal{O}]) \in \mathcal{C}_{\infty}(\mathcal{G},N)\] where $[\cdot]$ denotes the image in $\Lambda_N^{\pm}$.   
\end{proposition}

Observe that if $-\text{Id} \in \mathcal{G}$, then each $\mathcal{G}$-orbit is irregular in the sense defined above, though every cusp of $\mathcal{G}$ is regular. 
	
\begin{proof}
We choose $\gamma_{0} \in \Gamma$ so that $\mathcal{O} = \bar{\lambda}\mathcal{G}$ where $\bar{\lambda} = \overline{(0, 1)}\gamma_0$. Then $\text{or}_{\mathcal{G},N}^{-1}\big(\phi_N^{-1}([\mathcal{O}])\big) = \mathcal{G}(\gamma_{0}^{-1}\infty)$. Suppose $\gamma \in \mathcal{G}$. We have 
\begin{equation} 
\label{3.1}
\bar{\lambda}\gamma = - \bar{\lambda}\iff \overline{(0,1)}\gamma_{0}\gamma\gamma_{0}^{-1} =  \overline{(0, -1)}.
\end{equation}
Write $\gamma_{0}\gamma\gamma_{0}^{-1} = \begin{pmatrix}
			a(\gamma) & b(\gamma)\\
			c(\gamma) & d(\gamma)
		\end{pmatrix}$. 
Then \eqref{3.1} holds if and only if $c(\gamma) \equiv 0($mod $N)$ and $a(\gamma) \equiv d(\gamma) \equiv - 1 ($mod $N)$. Set \[\text{$g_1(\gamma) = \begin{pmatrix}
			- 1 & b(\gamma)\\
			0 & -1
		\end{pmatrix}$ and $g_2(\gamma) = g_1(\gamma)^{-1}\gamma_{0}\gamma\gamma_{0}^{-1}$.}\]
It is easy to see that \eqref{3.1} holds if and only if $g_2(\gamma) \in \Gamma(N)$. Let $\mathcal{O}$ be irregular, i.e., there exists $\gamma \in \mathcal{G}$ with $\bar{\lambda}\gamma = - \bar{\lambda}$. Then $g_2(\gamma) \in \Gamma(N)$ and $g_1(\gamma) \in \gamma_{0}\mathcal{G}\gamma_{0}^{-1}$. Since $-\text{Id} \notin \gamma_{0}\mathcal{G}\gamma_{0}^{-1}$, it follows that $\mathcal{G}(\gamma_{0}^{-1}\infty)$ is an irregular cusp of $\mathcal{G}$. Conversely, assume that $\mathcal{G}(\gamma_{0}^{-1}\infty)$ is an irregular cusp of $\mathcal{G}$. Then $\exists h \geq 1$ so that $- T^h \in \gamma_0\mathcal{G}\gamma_{0}^{-1}$. Put $\gamma = -\gamma_{0}^{-1} T^h \gamma_{0} \in \mathcal{G}$. It is clear that $\bar{\lambda}\gamma = - \bar{\lambda}$ and therefore $\mathcal{O} = -\mathcal{O}$.  
\end{proof} 

Now, we explain how to use regular orbits to discard the sign ambiguity appearing in the transformation formula \eqref{2.10new}. For this purpose, one needs to work with a special kind of choice of representatives for the $\{\pm \text{Id}\}$ action on $\Lambda_{N}$ that is consistent with the $\mathcal{G}$-action on $\Lambda_{N}$. A choice of representatives $\mathcal{A}$ for the $\{\pm \text{Id}\}$-action is $\mathcal{G}$-\textit{admissible} if each regular $\mathcal{G}$-orbit $\mathcal{O}$ in $\Lambda_{N}$ satisfies the following implication: \[\mathcal{O} \cap \mathcal{A} \neq \emptyset \implies \mathcal{O} \subseteq \mathcal{A}.\]
One can construct an admissible choice of representatives $\mathcal{A}$ as follows. 
Let $\Lambda_{N}^{\pm} = \amalg_{j \in \mathcal{J}} \theta_{j}\mathcal{G}$ be the orbit decomposition of $\Lambda_{N}^{\pm}$ under the action of $\mathcal{G}$. Choose $\bar{\lambda}_j \in \Pi_{N}^{-1}(\theta_j)$ for each $j \in \mathcal{J}$, where $\Pi_{N}: \Lambda_{N} \to \Lambda_{N}^{\pm}$ is the natural projection. Fix a set of representatives $\mathcal{A}$ as follows. If $\bar{\lambda}_j$ lies in a regular orbit for some $j \in \mathcal{J}$ then the representative of $\theta_{j}\gamma$ is $\bar{\lambda}_j\gamma$ for each $\gamma \in \mathcal{G}$. This is possible since $\Pi_{N}$ is $\Gamma$-equivariant and maps $\bar{\lambda}_j\mathcal{G}$ bijectively onto $\theta_j\mathcal{G}$. If $\bar{\lambda}_j$ lies in an irregular orbit, then we choose any element of $\Pi_{N}^{-1}(\theta)$ to be the lift of $\theta \in \theta_{j}\mathcal{G}$. The following lemma displays the utility of the concept of admissibility. 

\begin{lemma}
\label{lemma3.2new}
Assume that $-\emph{Id} \notin \mathcal{G}$ and let $\mathcal{A}$ be a $\mathcal{G}$-admissible set of representatives. Suppose $y \in \mathcal{C}\big(\Gamma(N)\big)$ with $\mathcal{G}y \in \mathcal{C}_{\infty}(\mathcal{G},N)$. Then $\Phi_{N, \mathcal{A}}(\gamma y) =\Phi_{N, \mathcal{A}}(y)\lvert_{\gamma^{-1}}$, i.e., $\varepsilon_{\gamma}(y) = 1$ for each $\gamma \in \mathcal{G}$. 
\end{lemma}
\begin{proof}
Follows from Proposition~\ref{proposition3.1new} and the discussion above. 
\end{proof}

Let $x \in \mathbb{P}^{1}(\mathbb{Q})$. Our next job is to determine the size of $\text{or}_{\mathcal{G},N}(x)$. Note that the stabilizer of $\Gamma(N)x$ for the action of $\mathcal{G}$ equals $\mathcal{G}_x\Gamma(N)$. As a consequence we have $\lvert \text{or}_{\mathcal{G},N}(x) \rvert = [\mathcal{G}: \mathcal{G}_x\Gamma(N)]$. For a subgroup $G$ of $\Gamma$, we write $G^{\pm}$ to denote the subgroup generated by $\{-\text{Id}, G\}$. Since $-\text{Id}$ acts trivially on the points of $\mathbb{P}^{1}(\mathbb{Q})$, it follows that $\lvert \text{or}_{\mathcal{G}, N} (x)\rvert = \lvert \text{or}_{\mathcal{G}^{\pm}, N} (x)\rvert$. We use the formula above to discover $\bigl\lvert \text{or}_{\mathcal{G}^{\pm}, N} (x)\bigr\rvert = \frac{[\Gamma: \Gamma_x\Gamma(N)]}{[\Gamma: \mathcal{G}^{\pm}]}[\Gamma_x: \mathcal{G}^{\pm}_x]$
where we are using the isomorphism $\frac{\Gamma_x}{\mathcal{G}_x^{\pm}} \cong \frac{\Gamma_x\Gamma(N)} {\mathcal{G}^{\pm}_x\Gamma(N)}$. Observe that $[\Gamma: \Gamma_x\Gamma(N)]$ does not depend on the choice of the point $x$. Thus, \[\lvert \text{or}_{\mathcal{G},N} (x) \rvert = c(\mathcal{G}) \text{amp}(x, \mathcal{G})\] 
where $c(\mathcal{G})$ a constant that is independent of $x$ and $\text{amp}(x, \mathcal{G}) := [\Gamma_x: \mathcal{G}_x^{\pm}]$ is the amplitude of $\mathcal{G}$ at the cusp $\mathcal{G}x$. Now, a theorem of Larcher \cite{larcher} asserts that the set of cusp amplitudes $\{\text{amp}(x, \mathcal{G}) \mid x \in \mathcal{C}(\mathcal{G})\}$ is closed under the operations g.c.d. and l.c.m. In particular, there exists a cusp $x_0$ so that $\text{amp}(x_0, \mathcal{G})$ divides $\text{amp}(x, \mathcal{G})$ for each $x \in \mathcal{C}(\mathcal{G})$. Our treatment makes use of this fact during the construction of a basis for the space of Eisenstein series of weight $2$. We conclude the discussion with a glimpse into the details of Larcher's theorem, which also allows us to describe the main construction of this section in a completely explicit manner for a large number of cases. 

Let $\mathcal{G}$ be a congruence subgroup of $\Gamma$ containing $-\text{Id}$. We say $\mathcal{K} \subseteq \mathcal{G}$ is a \textit{general Larcher subgroup} \cite{cummins} of $\mathcal{G}$ if $- \text{Id} \in \mathcal{K}$ and the subset of parabolic elements of $\mathcal{K}$ coincides with the subset of the parabolic elements of $\mathcal{G}$. If $\mathcal{K}$ is a general Larcher subgroup of $\mathcal{G}$, then the cusp amplitudes of $\mathcal{G}$ coincide with the cusp amplitudes of $\mathcal{K}$. For positive integers $p$, $q$, $r$, $\chi$, and $\tau$ such that $p \mid qr$ and $\chi \mid \gcd(p, \frac{qr}{p})$, write
\begin{equation}
\label{3.2new}
H(p,q,r; \chi, \tau) := \Big\{\begin{pmatrix}
	1 + ap & bq \\
	cr & 1 + dp
\end{pmatrix} \in \Gamma\; \bigl\lvert \; \substack{a, b, c, d \in \mathbb{Z}, \\ c \equiv \tau a (\text{mod $\chi$})}\Big\}.
\end{equation}
Larcher showed that every congruence subgroup $\mathcal{G}$ containing $-\text{Id}$ contains a general Larcher subgroup that is conjugate to a subgroup of the form $H(p,q,r; \chi, \tau)^{\pm}$ for a suitable choice of the parameters. It is possible to explicitly compute the orbits in $\Lambda_{N}$ for the class of groups in \eqref{3.2new}; see \cite[7]{cummins et al}. We list the orbits for the classical families $\Gamma(N,t)$ and $\Gamma_{0}(N)$ in Appendix~\ref{appendix}, necessary for the discussion in the latter sections. A general congruence subgroup need not be conjugate to a subgroup in the form given in \eqref{3.2new}. Nevertheless, if the image of $\mathcal{G}$ in $\mathsf{P}\Gamma := \Gamma/\{\pm \text{Id}\}$ is a torsion free subgroup of genus zero, then Sebbar proves that $\mathcal{G}^{\pm}$ is indeed conjugate to a subgroup of the form $H(p,q,r; \chi, \tau)^{\pm}$. In addition, one expects that these subgroups should play an important role in the classification of the subgroups of higher genus.

\section{Spectral Eisenstein series}
\label{section4}
\subsection{Definition and first properties}
\label{section4.1}
Let $k$ be an integer $\geq 2$ and $\mathcal{G}$ be a congruence subgroup of $\Gamma$. As per standard convention, we denote a cusp of $\mathcal{G}$ using a fixed choice of representative in $\mathbb{P}^{1}(\mathbb{Q})$. Suppose that $x \in \mathcal{C}(\mathcal{G})$ and $\mathcal{G}_{x}$ is the stabilizer $x$ in $\mathcal{G}$. We choose a \textit{scaling matrix} $\sigma_{x} \in \Gamma$ for $x$, i.e., $x = \sigma_x\infty$. The \textit{spectral Eisenstein series} of weight $k$ attached to the cusp $x$ and the scaling matrix $\sigma_x$ is 
\begin{equation}
\label{4.1}
E_{k,x}(\tau;s) := \sum_{\gamma \in \mathcal{G}_x\backslash\mathcal{G}} \mathbf{j}(\sigma_x^{-1}\gamma, \tau)^{-k}(\text{Im } \sigma_x^{-1}\gamma\tau)^{s} \hspace{.3cm}\substack{(\text{Re}(k+2s) > 2)}
\end{equation}
where $\mathbf{j}\big(\gamma, \tau\big) = c\tau+d$ for $\gamma = \begin{pmatrix}
	a & b\\
	c & d
\end{pmatrix} \in \Gamma$. Note that $\mathbf{j}$ satisfies the cocycle identity   
\[\mathbf{j}(\gamma_1\gamma_2, \tau) = \mathbf{j}(\gamma_1, \gamma_2\tau)\mathbf{j}(\gamma_2, \tau). \hspace{.3cm}\substack{(\gamma_1, \gamma_2 \in \Gamma)}\]
Thus, if $k$ is even then the sum on the RHS of \eqref{4.1} does not depend on a choice of representatives for $\mathcal{G}_x \backslash \mathcal{G}$, and gives rise to a well-defined function that is invariant under $\lvert_{k}$ action of $\mathcal{G}$. Now suppose $k$ is odd. Here the series \eqref{4.1} may depend on a choice of representatives for $\mathcal{G}_x \backslash \mathcal{G}$ unless $-\text{Id} \notin \mathcal{G}$ and $x$ is a regular cusp of $\mathcal{G}$. If $-\text{Id} \notin \mathcal{G}$ and $x$ is a regular cusp of $\mathcal{G}$, then the RHS of \eqref{4.1} gives rise to well-defined $\mathcal{G}$-invariant function on $\mathbb{H}$. Next, we analyze the dependence of the spectral Eisenstein series on the choice of scaling matrix. Let the notation be as above and if $k$ is odd then assume that $-\text{Id} \notin \mathcal{G}$ and $x$ is a regular cusp. If $\sigma_{x, 1}$ and $\sigma_{x,2}$ are two scaling matrices for the cusp $x$ then $\sigma_{x,2}^{-1} \in \Gamma_{\infty} \sigma_{x, 1}^{-1}$. Now, an application of the cocycle identity shows that the series attached to $\sigma_{x, 1}$ and $\sigma_{x,2}$ are equal up to a sign factor of the form $(\pm 1)^k$. In particular, an Eisenstein series of even weight is independent of the choice of scaling matrix. But, if $k$ is odd, then replacing $\sigma_{x}$ by $-\sigma_{x}$ replaces the series by its negative. In this article, we always write the Eisenstein series with respect to a choice of scaling matrix. Finally, let $y$ be another representative for the cusp $\mathcal{G}x$. Write $y = \gamma_{0} x$ with $\gamma_{0} \in \mathcal{G}$. Note that there is a bijection 
$\mathcal{G}_x \backslash \mathcal{G} \xrightarrow{\cong} \mathcal{G}_y \backslash \mathcal{G}$ described by $\mathcal{G}_x\gamma \mapsto \mathcal{G}_y \gamma_{0} \gamma$.
Therefore the series attached to the cusp $x$ and scaling matrix $\sigma_{x}$ is equal to the series attached to $y$ and scaling matrix $\gamma_{0}\sigma_{x}$. Thus the Eisenstein series \eqref{4.1} is essentially independent of the choice of representative for the cusp.

We begin by comparing the construction above with the classical Eisenstein series on $\Gamma(N)$. First, let $\mathcal{G} = \Gamma(N)$, and if $k$ is odd, then additionally assume that $N \geq 3$. Let $x$ be cusp of $\Gamma(N)$ with scaling matrix $\sigma_{x}$. Then a direct calculation \cite[p.111]{ds} comparing the defining expressions demonstrates 
\begin{equation}
\label{4.2}
E_{k,x}(\tau;s) = E_{k}\big(\overline{(0,1)}\sigma_{x}^{-1}, N; s\big). \hspace{.3cm}(\substack{\text{Re}(k+2s) > 2})
\end{equation}
Observe that, in this situation $\phi_{N}(x) = [\overline{(0,1)}\sigma_{x}^{-1}]$ where $\phi_{N}$ is the map defined in \eqref{2.1new}. One can enhance the map $\Phi_{N, \mathcal{A}}$ in Section~\ref{section2} to a spectral version by setting  
\begin{equation}
\label{4.3}
\Phi_{N, \mathcal{A}}(y;s) := E_{k}\big(\tau, \Pi_{N, \mathcal{A}}^{-1}\circ \phi_{N}(y), N;s \big). \hspace{.3cm}\big(\substack{\text{Re}(k+2s) > 2,\; y \in \mathcal{C}(\Gamma(N))}\big) 
\end{equation}
It is easy to see that the transformation properties of $\Phi_{N, \mathcal{A}}(\bullet)$, namely \eqref{2.10new} and Lemma~\ref{lemma3.2new}, carry over to $\Phi_{N, \mathcal{A}}(\bullet; s)$ without any change. 
The following result utilizes this construction to express the Eisenstein series on a general congruence subgroup \eqref{4.1} in terms of the normalized series on a principal level contained in it. Recall the notion of $\mathcal{G}$-admissible choice of representatives for the $\{\pm \text{Id}\}$-action on $\Lambda_{N}$ from Section~\ref{section3}.

\begin{proposition}
\label{proposition4.1new} 
Let $\mathcal{G}$ be a congruence subgroup containing $\Gamma(N)$. Suppose that $s \in \mathbb{C}$ with $\emph{Re}(k+2s) > 2$. 
\begin{enumerate}[label=(\roman*), align=left, leftmargin=0pt]
\item Let $k$ be even. Then 
\begin{equation*}
E_{k,x}(\tau;s) = \sum_{y \in \emph{or}_{\mathcal{G},N}(x)} \Phi_{N, \mathcal{A}}(y;s), \hspace{.3cm}\forall x \in \mathcal{C}(\mathcal{G}).
\end{equation*}
\item Now suppose $k$ is odd and $-\emph{Id} \notin \mathcal{G}$. Given a  $\mathcal{G}$-admissible choice of representatives $\mathcal{A}$ there exists a choice of scaling matrix $\{\sigma_{x} \mid x \in \mathcal{C}_{\infty}(\mathcal{G})\}$ so that 
\begin{equation}
\label{4.4}
E_{k,x}(\tau;s) = \sum_{y \in \emph{or}_{\mathcal{G},N}(x)} \Phi_{N, \mathcal{A}}(y;s), \hspace{.3cm}\forall x \in \mathcal{C}_{\infty}(\mathcal{G}).
\end{equation}
Conversely, the choice of scaling matrix for each regular cusp determines a choice of $\mathcal{G}$-admissible representatives $\mathcal{A}$ so that the above identity holds. 
\end{enumerate}
\end{proposition}

Observe that the transformation properties mentioned above ensure the RHS of the identities in the statement are already $\mathcal{G}$-invariant. 

\begin{proof}
Let $x \in \mathcal{C}(\mathcal{G})$. Suppose that $\{\gamma_1, \ldots \gamma_n\}$ is a set of representatives for the coset $\mathcal{G}/\mathcal{G}_x\Gamma(N)$. Then $\text{or}_{\mathcal{G},N}(x) = \{\gamma_{j}x \mid 1 \leq j \leq n\}$ and $\mathcal{G}_{x} \backslash \mathcal{G} = \coprod_{j} \Gamma(N)_{x}\backslash \Gamma(N)\gamma_{j}^{-1}$. If $\sigma_{x}$ is a scaling matrix for $x$, then $\gamma_{j} \sigma_{x}$ is a scaling matrix for $\gamma_{j}x$. For this choice of scaling parameter, we have 
\begin{equation*}
E_{k,x, \mathcal{G}}(\tau;s) = \sum_{j = 1}^{n} E_{k, \gamma_{j}x, \Gamma(N)}(\tau;s)	
\end{equation*}
whenever the LHS of the equality is well-defined. Here, the extra subscript records the group of definition for the corresponding Eisenstein series. The first part of the assertion is a direct consequence of \eqref{4.2} and the identity above. Assume that $k$ is odd and $- \text{Id} \notin \mathcal{G}$. Let $\{\sigma_{x} \mid x \in \mathcal{C}_{\infty}(\mathcal{G})\}$ be a choice of scaling matrix for the regular cusps. For each $x \in \mathcal{C}_{\infty}(\mathcal{G})$ we fix a choice of representatives $\{\gamma_{j, x} \mid 1 \leq j \leq n(x)\} \subseteq \mathcal{G}$ for $\mathcal{G}/\mathcal{G}_x\Gamma(N)$ as in the beginning of the proof.  Then any choice of representatives $\mathcal{A}$ containing 
\[\{\overline{(0,1)}\sigma_{x}^{-1}\gamma_{j, x}^{-1} \mid x \in \mathcal{C}_{\infty}(\mathcal{G}), \; 1 \leq j \leq n(x)\}\] 
is automatically $\mathcal{G}$-admissible and satisfies \eqref{4.4}. Conversely, suppose that $\mathcal{A}$ is a given $\mathcal{G}$-admissible choice of representative. Let $\sigma_{x} \in \Gamma$ be such that $\Pi_{N, \mathcal{A}}^{-1} \circ \phi_{N}(x) = \overline{(0,1)}\sigma_{x}^{-1}$. Then $\{\sigma_x \mid x \in \mathcal{C}_{\infty}(\mathcal{G})\}$ is choice of scaling matrices so that \eqref{4.4} holds. 
\end{proof}

Proposition~\ref{proposition4.1new} provides an efficient method to study the spectral Eisenstein series in terms of the normalized Eisenstein series. For example, the properties of the $E$-series in Section~\ref{section2} implies that the spectral Eisenstein series \eqref{4.1} analytically continues to the larger half plane $\text{Re}(k+2s) > 1$. Hence we can specialize the spectral Eisenstein series at $(k,s) = (2,0)$ exactly like the $E$-series. The rest of the current article only concerns the specialization of \eqref{4.1} at $s =0$. For simplicity, one abbreviates $E_{k,x}(\tau;0)$ as $E_{k,x}(\tau)$. The Fourier expansion formulas in Section~\ref{section2} readily shows that 
\begin{equation}
\label{4.5}
E_{k,x}(\tau) \in \begin{cases}
	\text{Hol}(\mathbb{H}), & \text{if $k \geq 3$;}\\
	- \frac{6 \epsilon_{N} \lvert \text{or}_{\mathcal{G},N}(x) \rvert}{\pi N^2 \prod_{p \mid N} (1 - \frac{1}{p^2}) \text{Im}(\tau)} + \text{Hol}(\mathbb{H}), & \text{if $k=2$} 
\end{cases}
\end{equation}
whenever LHS is well-defined and $\Gamma(N) \subseteq \mathcal{G}$. We next use the constant term formula \eqref{2.8new} to verify that the constant term of $E_{k,x}(\tau)$ possesses the desired indicator of cusp property.  

\begin{lemma}
\label{lemma4.2new}
Let $\mathcal{G}$ be a congruence subgroup and $x \in \mathcal{C}(\mathcal{G})$. If $k$ is odd, then assume that $-\emph{Id} \notin \mathcal{G}$, $x$ is regular, and fix a scaling matrix for $x$. We have 
\[\pi_{\infty}\big(E_{k,x}\lvert_{\gamma}\big) = \begin{cases}
	(\pm 1)^k, & \text{if $\gamma \infty = x$ in $\mathcal{C}(\mathcal{G})$;}\\
	0, & \text{otherwise.}
\end{cases} \hspace{.5cm}\substack{(\forall \gamma \in \Gamma)}\] 
\end{lemma}

\begin{proof}
Let $N$ be a positive integer with $\Gamma(N) \subseteq \mathcal{G}$. We use the formulas in Proposition~\ref{proposition4.1new} together with \eqref{2.8new} to deduce  
\[\pi_{\infty}\big(E_{k,x}\lvert_{\gamma}\big) = \begin{cases}
	(\pm 1)^k, & \text{if $\infty_N \in \{\gamma^{-1}y \mid y \in \text{or}_{\mathcal{G},N}(x)\}$;}\\
	0, & \text{otherwise.}
\end{cases} \hspace{.5cm}\substack{(\forall \gamma \in \Gamma)}\] 
But, the first condition above holds if and only if $\gamma \infty = x$ in $\mathcal{C}(\mathcal{G})$.   
\end{proof}

We conclude the discussion with a convenient rephrasing of Proposition~\ref{proposition4.1new}, which is helpful for practical calculations. Let the notation be as in the statement of the proposition. Suppose that $\mathcal{O}_{\text{reg}}$, resp. $\mathcal{O}_{\text{ir}}$, is the collection of regular, resp. irregular, $\mathcal{G}$-orbits in $\Lambda_{N}$. With each $\mathcal{G}$-orbit  $\mathcal{O}$ in $\Lambda_{N}$ one can attach a cusp $x_{\mathcal{O}}$ in $\mathcal{C}(\mathcal{G})$ by setting $x_{\mathcal{O}} := \text{or}_{\mathcal{G},N}^{-1} \circ \phi_{N}^{-1}\big([\mathcal{O}]\big)$ where $[\cdot]$ refers to the image in $\Lambda_{N}^{\pm}$. Proposition~\ref{proposition3.1new} shows that $\mathcal{O} \in \mathcal{O}_{\text{reg}}$ if and only if $x_{\mathcal{O}} \in \mathcal{C}_{\infty}(\mathcal{G})$ whenever  $-\text{Id} \notin \mathcal{G}$. The natural map $\mathcal{O} \to \phi_{N}^{-1}([\mathcal{O}])$ described by $\bar{\lambda} \mapsto \phi_{N}^{-1}([\bar{\lambda}])$ is a bijection if $\mathcal{O} \in \mathcal{O}_{\text{reg}}$ and a two-to-one surjection if $\mathcal{O} \in \mathcal{O}_{\text{ir}}$ and $N \geq 3$. In the exceptional cases, i.e., for $\mathcal{O} \in \mathcal{O}_{\text{ir}}$ and $N \in \{1,2\}$ the natural map is again a bijection. Note that for a $\mathcal{G}$-orbit $\mathcal{O}$ in $\Lambda_{N}$, one can construct a $\mathcal{G}$-invariant nonholomorphic Eisenstein series by considering the orbital sum:  
\begin{equation}
\label{4.6new}
E_{k}(\mathcal{O}, \mathcal{G}; s) := \sum_{\bar{\lambda} \in \mathcal{O}} E_{k}(\bar{\lambda}, N; s). \hspace{.3cm}(\substack{\text{Re}(k+2s) > 2})
\end{equation}
If $k$ is odd and $-\text{Id} \in \mathcal{G}$ then this sum is zero. Assume that $-\text{Id} \notin \mathcal{G}$ if $k$ is odd. Let $\mathcal{A}$ be a choice of $\mathcal{G}$-admissible representatives for the action of $\{\pm \text{Id}\}$ on $\Lambda_N$ and consider a compatible choice of scaling matrices in the sense of Proposition~\ref{proposition4.1new}. Then 
\begin{equation*}
E_{k}(\mathcal{O}, \mathcal{G}; s) = \begin{cases}
	E_{k, x_{\mathcal{O}}}(\tau;s), & \text{if $\mathcal{O} \in \mathcal{O}_{\text{reg}}$ and $\mathcal{O} \subseteq \mathcal{A}$};\\
	(-1)^{k}E_{k, x_{\mathcal{O}}}(\tau;s), & \text{if $\mathcal{O} \in \mathcal{O}_{\text{reg}}$ and $-\mathcal{O} \subseteq \mathcal{A}$};\\
	2E_{k, x_{\mathcal{O}}}(\tau;s), & \text{if $\mathcal{O} \in \mathcal{O}_{\text{ir}}$, $k$ is even, and $N \geq 3$};\\
	E_{k, x_{\mathcal{O}}}(\tau;s), & \text{if $\mathcal{O} \in \mathcal{O}_{\text{ir}}$, $k$ is even, and $N= 1,2$};\\
	0, & \text{otherwise}.  	
\end{cases}	
\end{equation*}
In other words, the orbital sums \eqref{4.6new} practically capture the same information as the spectral Eisenstein series \eqref{4.1}. 

\subsection{Basis for the space of Eisenstein series}
\label{section4.2}
The spectral Eisenstein series already exhibits an appropriate number of linearly independent functions inside the space of Eisenstein series. From here one can invoke the dimension formula for the space of Eisenstein series to check that they form a basis of the space. However, we prefer to avoid using the dimension formula and utilize the fixed subspace description provided in \eqref{1.1} to finish the construction of the basis. The discussion begins with a few auxiliary notions that streamline the presentation and help us keep track of the coefficients during the procedure. 

Let $k \geq 2$ and $\mathbbm{k}$ be a subring of $\mathbb{C}$. Suppose that $\mathcal{G}$ is a congruence subgroup containing $\Gamma(N)$. We define the extended space of Eisenstein series on $\mathcal{G}$ with coefficients in $\mathbbm{k}$ as 
\begin{equation*}
\mathcal{E}^{\ast}_{k}(\mathcal{G}, \mathbbm{k}) := \mathcal{E}^{\ast}_{k}\big(\Gamma(N), \mathbbm{k}\big)^{\mathcal{G}}.
\end{equation*}
Moreover, the space of Eisenstein series on $\mathcal{G}$ with the coefficients in $\mathbbm{k}$ is    
\begin{equation}
\label{4.6}
\mathcal{E}_{k}(\mathcal{G}, \mathbbm{k}) := \mathcal{E}^{\ast}_{k}(\mathcal{G}, \mathbbm{k}) \cap \text{Hol}(\mathbb{H}) = \mathcal{E}_{k}\big(\Gamma(N),\mathbbm{k}\big)^{\mathcal{G}}.
\end{equation}
The two spaces of Eisenstein series are equal if $k \geq 3$. The results of this subsection imply that these spaces are independent of $\Gamma(N)$ and the latter space coincides with the $\mathbbm{k}$-span of the spectral basis yet to be constructed. Note that $\mathcal{E}_{k}(\mathcal{G}, \mathbb{C})$ is the same as $\mathcal{E}_{k}(\mathcal{G})$ described in \eqref{1.1} since $\mathcal{E}_{k}\big(\Gamma(N),\mathbb{C}\big)$ is the correct space of Eisenstein series on $\Gamma(N)$. The construction of a basis for the even weight spaces is a straightforward consequence of the formalism developed in Section~\ref{section2}. 

\begin{lemma}
\label{lemma4.3new}
Let $k$ be even and $E_{k,x}$ denote the spectral Eisenstein series from Section~\ref{section4.1}. Then, the functions $\{E_{k,x} \mid x \in \mathcal{C}(\mathcal{G})\}$ form a $\mathbbm{k}$-basis of the space $\mathcal{E}^{\ast}_{k}(\mathcal{G}, \mathbbm{k})$. 
\end{lemma}
\begin{proof}
The transformation formula \eqref{2.10new} shows that $\Phi_{N, \mathcal{A}}$ is $\Gamma$-equivariant, i.e., $\Phi_{N, \mathcal{A}}(\gamma y) = \Phi_{N, \mathcal{A}}(y)\gamma^{-1}$ for each $y \in \mathcal{C}(\Gamma(N))$ and $\gamma \in \Gamma$. Therefore, a standard argument regarding the permutation representations demonstrates that $\big\{\sum_{\small{y \in  \text{or}_{\mathcal{G},N}(x)}} \Phi_{N, \mathcal{A}}(y) \; \big\lvert \; x \in \mathcal{C}(\mathcal{G})\big\}$ is a $\mathbbm{k}$-basis of the extended space of Eisenstein series. Thus the assertion follows from Proposition~\ref{proposition4.1new}(i).  
\end{proof}

In the light of \eqref{4.5} one discovers that the holomorphic subspace of weight $2$ is \[\mathcal{E}_{2}(\mathcal{G}, \mathbbm{k}) = \big\{\sum_{x \in \mathcal{C}(\mathcal{G})} C_{x}E_{2,x}\,\bigl\lvert\, \sum_{x} C_{x}\lvert \text{or}_{\mathcal{G},N}(x)\rvert = 0\big\}.\] 
We use Larcher's theorem to choose $x_0 \in \mathbb{P}^{1}(\mathbb{Q})$ so that $\text{amp}(x_0, \mathcal{G})$ divides $\text{amp}(x, \mathcal{G})$ for each $x \in \mathcal{C}(\mathcal{G})$. Then, by the discussion in Section~\ref{section3}, the collection 
\begin{equation}
\label{4.7new}
\big\{E_{2,x} - \frac{\lvert \text{or}_{\mathcal{G},N}(x) \rvert}{\lvert \text{or}_{\mathcal{G},N}(x_0) \rvert}E_{2,x_0}\,\bigl\lvert\, x \in \mathcal{C}(\mathcal{G}) - \{x_0\} \big\}
\end{equation}
is a $\mathbbm{k}$-basis of $\mathcal{E}_{2}(\mathcal{G}, \mathbbm{k})$. Finally, let $k$ be an odd integer and assume that $-\text{Id} \notin \mathcal{G}$. This assumption automatically forces $N \geq 3$.

\begin{proposition}
\label{proposition4.4new}
Let the notation be as above and fix a choice of scaling matrix for each regular cusp of $\mathcal{G}$. Then $\{E_{k,x} \mid x \in \mathcal{C}_{\infty}(\mathcal{G})\}$ is a $\mathbbm{k}$-basis of $\mathcal{E}_{k}(\mathcal{G},\mathbbm{k})$.
 \end{proposition}

\begin{proof}
The constant term formula in Lemma~\ref{lemma4.2new} ensures the given collection of functions is already linearly independent. Let $\mathcal{A}$ be a set of $\mathcal{G}$-admissible representatives compatible with our choice of scaling matrices in the sense of Proposition~\ref{proposition4.1new}. Suppose that $f \in \mathcal{E}_{k}(\mathcal{G},\mathbbm{k})$ and write $f = \sum_{y \in \mathcal{C}(\Gamma(N))} \alpha_{y} \Phi_{N, \mathcal{A}}(y)$. Then the $\mathcal{G}$-invariance of $f$ and \eqref{2.10new} together imply that 
\begin{equation*}
	\alpha_{y}\varepsilon_{\gamma}(y)^k = \alpha_{\gamma y};\hspace{.5cm} \forall y \in \mathcal{C}(\Gamma(N)), \gamma \in \mathcal{G}.     
\end{equation*}
Let $y \in \mathcal{C}\big(\Gamma(N)\big)$ be so that $\mathcal{G}y \in \mathcal{C}_{\infty}(\mathcal{G},N)$. Then $\alpha_{y} = \alpha_{\gamma y}$ for each $\gamma \in \mathcal{G}$ (Lemma~\ref{lemma3.2new}).  
Suppose $y \in \mathcal{C}\big(\Gamma(N)\big)$ with $\mathcal{G}y \in \mathcal{C}(\mathcal{G}) - \mathcal{C}_{\infty}(\mathcal{G})$. Write $\phi_{N}(y) = [\bar{\lambda}]$. Now Proposition~\ref{proposition3.1new} provides $\gamma \in \mathcal{G}$ satisfying $\bar{\lambda}\gamma^{-1} = - \bar{\lambda}$. Hence $\alpha_{y} = 0$ as desired. 
\end{proof}

The following theorem summarizes our construction of a basis. 

\begin{theorem}
\label{theorem4.5new}
Let $\mathcal{G}$ be a congruence subgroup containing $\Gamma(N)$. Then, the collection of functions 
\begin{equation*}
\mathcal{B}_{\mathcal{G}, k} =
\begin{cases}
\{E_{k,x} \mid x \in \mathcal{C}(\mathcal{G})\}, & \substack{\text{if $k$ is even and $\geq 4$;}}\\
\{E_{k,x} \mid x \in \mathcal{C}_{\infty}(\mathcal{G})\}, & \substack{\text{if $k$ is odd and $-\emph{Id} \notin \mathcal{G}$;}}\\
\{E_{k,x} - \frac{\lvert \emph{or}_{\mathcal{G},N}(x)\rvert}{\lvert \emph{or}_{\mathcal{G},N}(x_0)\rvert} E_{k,x_0}\mid x \in \mathcal{C}(\mathcal{G}) - \{x_0\}\}, & \substack{\text{if $k=2$;}}\\
\emptyset, & \substack{\text{otherwise}} 
\end{cases}
\end{equation*}
forms a $\mathbbm{k}$-basis of $\mathcal{E}_{k}(\mathcal{G},\mathbbm{k})$. Here, in the second case, we write the series for a fixed choice of scaling matrix for each regular cusp.
\end{theorem}

Letting $\mathbbm{k} = \mathbb{C}$ one discovers that $\mathcal{B}_{\mathcal{G},k}$ is a $\mathbb{C}$-basis of $\mathcal{E}_k(\mathcal{G})$. We refer to this basis as the \textit{spectral basis} of the space of Eisenstein series on $\mathcal{G}$. The computations of this subsection demonstrate that $\mathcal{E}_{k}(\mathcal{G}, \mathbbm{k})$, as defined in \eqref{4.6}, equals the free $\mathbbm{k}$-submodule spanned by the spectral basis attached to $\mathcal{G}$. Counting the size of $\mathcal{B}_{\mathcal{G}, k}$ one retrieves the dimension of the space of Eisenstein series on $\mathcal{G}$ as promised in the introduction.  

\begin{corollary}
\label{corollary4.6new}
We have 
\[\emph{dim}_{\mathbb{C}} \mathcal{E}_k(\mathcal{G}) = \begin{cases}
	\lvert \mathcal{C}(\mathcal{G}) \rvert, & \text{if $k$ is even and $\geq 4$;}\\
	\lvert \mathcal{C}_{\infty}(\mathcal{G}) \rvert, & \text{if $k$ is odd and $-\emph{Id} \notin \mathcal{G}$;}\\
	\lvert \mathcal{C}(\mathcal{G}) \rvert -1, & \text{if $k = 2$;}\\
	0, & \text{otherwise.}
\end{cases} \] 
\end{corollary}

\section{Unnormalized series on general subgroups}
\label{section5}
\subsection{Eisenstein series using the $G$-series}
\label{section5.1}
Let $\mathcal{G}$ be a congruence subgroup containing $\Gamma(N)$. Suppose that $k$ is an integer $\geq 2$. We first introduce an unnormalized cousin of the function $\Phi_{N, \mathcal{A}}(\bullet, s)$ in Section~\ref{section4.1}. Let the notation be as in \eqref{4.3}. Write  
\[\Phi^{\text{un}}_{N, \mathcal{A}}(y;s) := G_{k}\big(\tau, \Pi_{N, \mathcal{A}}^{-1}\circ \phi_{N}(y), N;s \big). \hspace{.3cm}\big(\substack{\text{Re}(k+2s) > 2,\; y \in \mathcal{C}(\Gamma(N))}\big) \]
The transformation properties of $\Phi_{N, \mathcal{A}}(\bullet, s)$ merely record the behavior of the underlying parameter space under the group action. Therefore they carry over to the realm of $\Phi^{\text{un}}_{N, \mathcal{A}}(\bullet;s)$ without any change. If $k$ is odd then assume that $-\text{Id} \notin \mathcal{G}$, $\mathcal{A}$ is a $\mathcal{G}$-admissible choice of representatives, and $x$ is a regular cusp of $\mathcal{G}$. One defines the \textit{unnormalized} Eisenstein series on $\mathcal{G}$ as 
\begin{equation}
\label{5.1}
G_{k,x}(\tau;s) : = \sum_{y \in \text{or}_{\mathcal{G},N}(x)} \Phi^{\text{un}}_{N, \mathcal{A}}(y;s). \hspace{.3cm}(\substack{\text{Re}(k+2s) > 2})
\end{equation} 
The transformation formulas discussed above demonstrate that $G_{k,x}(\tau;s)$ is invariant under $\lvert_{k}$-action of $\mathcal{G}$. Note that the series \eqref{5.1} may depend on the choice of the auxiliary level $\Gamma(N)$. The analytic continuation of the $G$-series on $\Gamma(N)$ implies that the unnormalized Eisenstein series on $\mathcal{G}$ also analytically continues to the whole $s$-plane. We next generalize \eqref{2.4new} to write down explicit identities relating the two versions of Eisenstein series on a general congruence subgroup. 

The group $(\mathbb{Z}/N\mathbb{Z})^{\times}$ acts on $\Lambda_{N}$ through scalar multiplication, namely, $\bar{\delta} \cdot \bar{\lambda} = \bar{\delta} \bar{\lambda}$, which descends to a $(\mathbb{Z}/N\mathbb{Z})^{\times}$-action on $\mathcal{C}(\mathcal{G},N)$. We pullback this action using $\text{or}_{\mathcal{G},N}$ to obtain a well-defined action of $(\mathbb{Z}/N\mathbb{Z})^{\times}$ on $\mathcal{C}(\mathcal{G})$. Note that the $(\mathbb{Z}/N\mathbb{Z})^{\times}$-action on $\Lambda_{N}$ maps a $\mathcal{G}$-regular orbit to another $\mathcal{G}$-regular orbit. Therefore the induced action on $\mathcal{C}(\mathcal{G})$ leaves the collection $\mathcal{C}_{\infty}(\mathcal{G})$ stable; see Proposition~\ref{proposition3.1new}. 

\begin{proposition}
\label{proposition5.1}
Let the notation be as in \eqref{5.1} and assume that $\emph{Re}(k+2s) > 1$. Then 
\begin{equation}
\label{5.2}
\begin{gathered}
G_{k,x}(\tau;s) = \epsilon_{N}^{-1} \sum_{\bar{\delta} \in (\mathbb{Z}/N\mathbb{Z})^{\times}} \varepsilon_{\mathcal{A}}(\bar{\delta})^k \zeta_{+}^{(N)}(\delta, k+2s) E_{k,\bar{\delta}^{-1}x}(\tau;s),\\
E_{k,x}(\tau;s) = \epsilon_{N} \sum_{\bar{\delta} \in (\mathbb{Z}/N\mathbb{Z})^{\times}} \varepsilon_{\mathcal{A}}(\bar{\delta})^{k} \zeta_{+}^{(N)}(\mu; \delta, k+2s) G_{k,\bar{\delta}^{-1}x}(\tau;s)
\end{gathered}
\end{equation}
where $\varepsilon_{\mathcal{A}}(\bar{\delta}) \in \{\pm 1\}$. 
\end{proposition}  

\begin{proof}
We verify the identity on the smaller domain $\text{Re}(k+2s) > 2$ and deduce the general statement using the principle of analytic continuation. Let $\text{Re}(k+2s) > 2$ and $\mathcal{O}$ be a $\mathcal{G}$-orbit in $\Lambda_N$. Define an unnormalized analogue of the nonholomorphic orbital sums by $G_{k}(\mathcal{O}, \mathcal{G}; s) := \sum_{\bar{\lambda} \in \mathcal{O}} G_{k}(\bar{\lambda}, N; s)$ exactly as in \eqref{4.6new}. One computes using \eqref{2.4new} to deduce that  
\begin{equation}
\label{5.3}
\begin{gathered}
G_{k}(\mathcal{O}, \mathcal{G}; s) =\epsilon_{N}^{-1} \sum_{\bar{\delta} \in (\mathbb{Z}/N\mathbb{Z})^{\times}}  \zeta_{+}^{(N)}(\delta, k+2s) E_{k}(\bar{\delta}^{-1}\mathcal{O}, \mathcal{G};s),\\
E_{k}(\mathcal{O}, \mathcal{G}; s) =\epsilon_{N} \sum_{\bar{\delta} \in (\mathbb{Z}/N\mathbb{Z})^{\times}}  \zeta_{+}^{(N)}(\mu;\delta, k+2s) G_{k}(\bar{\delta}^{-1}\mathcal{O}, \mathcal{G};s).
\end{gathered}
\end{equation}
With a cusp $x \in \mathcal{C}(\mathcal{G})$ one associates a $\mathcal{G}$-orbit in $\Lambda_{N}$ by setting $\mathcal{O}_{x} = \bar{\lambda}_x\mathcal{G}$ where $\bar{\lambda}_x = \Pi_{N, \mathcal{A}}^{-1} \circ \phi_{N}(x)$. Put $\varepsilon_{\mathcal{A}}(\bar{\delta}) = 1$ if $\bar{\delta}^{-1}\mathcal{O}_x = \mathcal{O}_{\bar{\delta}^{-1}x}$ and $-1$ otherwise. Now, the identity \eqref{5.2} is a consequence of \eqref{5.3}. 
\end{proof}

As before we abbreviate $\Phi^{\text{un}}_{N,\mathcal{A}}(y;0)$, resp. $G_{k,x}(\tau;0)$, as $\Phi^{\text{un}}_{N,\mathcal{A}}(y)$, resp. $G_{k,x}(\tau)$. Observe that 
\begin{equation*}
G_{k,x}(\tau) \in \begin{cases}
		\text{Hol}(\mathbb{H}), & \text{if $k \geq 3$;}\\
		- \frac{\pi \lvert \text{or}_{\mathcal{G},N}(x) \rvert}{N^2 \text{Im}(\tau)} + \text{Hol}(\mathbb{H}), & \text{if $k=2$.} 
	\end{cases}
\end{equation*}
For a subring $\mathbbm{k}$ of $\mathbb{C}$, set 
\begin{equation*}
\begin{gathered}
\mathcal{E}^{*}_{k,\text{un}}\big(\Gamma(N), \mathbbm{k}\big) := \text{$\mathbbm{k}$-span of $\{G_k(\bar{\lambda},N) \mid \bar{\lambda} \in \Lambda_N\}$},\\
\mathcal{E}_{k,\text{un}}\big(\Gamma(N), \mathbbm{k}\big) := \mathcal{E}^{*}_{k,\text{un}}\big(\Gamma(N), \mathbbm{k}\big) \cap \text{Hol}(\mathbb{H}). 
\end{gathered} \hspace{.5cm}\substack{(k \geq 2)}
\end{equation*}
These subspaces are stable under the $\lvert_{k}$ action of $\Gamma$. One employs \eqref{2.4new} to discover that 
\begin{equation*}
\text{$\mathcal{E}^{*}_{k,\text{un}}\big(\Gamma(N), \mathbb{C}\big) = \mathcal{E}^{\ast}_{k}\big(\Gamma(N), \mathbb{C}\big)$ and $\mathcal{E}_{k,\text{un}}\big(\Gamma(N), \mathbb{C}\big) =\mathcal{E}_{k}\big(\Gamma(N), \mathbb{C}\big)$},
\end{equation*} i.e., both the constructions yield the same spaces over the field of complex numbers. With the notation of \eqref{2.9new}, the collection 
\begin{equation}
\label{5.4}
\{\Phi_{N, \mathcal{A}}^{\text{un}}(x) \mid x \in \mathcal{C}\big(\Gamma(N)\big)\} 
\end{equation}
is a spanning set of $\mathcal{E}^{\ast}_{k,\text{un}}\big(\Gamma(N), \mathbbm{k}\big)$. Moreover, this collection must be linearly independent over $\mathbb{C}$ since the dimension of $\mathcal{E}_{k,\text{un}}^{\ast}\big(\Gamma(N), \mathbb{C}\big)$, by the identity above, equals $\lvert \mathcal{C}\big(\Gamma(N)\big) \rvert$. In particular, $\mathcal{E}^{\ast}_{k,\text{un}}\big(\Gamma(N), \mathbbm{k}\big)$ is a free $\mathbbm{k}$-module generated by \eqref{5.4}.  One uses the explicit Fourier expansion \eqref{2.6new} to see that an analogue of \eqref{2.11new} holds in this setting and \[\{\Phi_{N, \mathcal{A}}^{\text{un}}(x) - \Phi_{N, \mathcal{A}}^{\text{un}}(\infty_N) \mid x \neq \infty_N\}\] is a basis of the holomorphic subspace $\mathcal{E}_{2,\text{un}}\big(\Gamma(N), \mathbbm{k}\big)$. Recall that the set of representatives $\mathcal{A}$ determines the sign factor $\varepsilon_{\gamma}(x)$ appearing in \eqref{2.10new}. As a consequence the same sign factor works for both $\Phi_{N, \mathcal{A}}$ and $\Phi_{N, \mathcal{A}}^{\text{un}}$. In particular, the assignment $\Phi_{N, \mathcal{A}}(x) \to \Phi_{N, \mathcal{A}}^{\text{un}}(x)$ is compatible with the action of $\Gamma$ and extends to a $\mathbbm{k}$-linear, $\Gamma$-equivariant isomorphism:  
\[\Psi_{N, \mathcal{A}}: \mathcal{E}^{\ast}_{k}\big(\Gamma(N), \mathbbm{k}\big) \xrightarrow{\cong} \mathcal{E}_{k,\text{un}}^{\ast}\big(\Gamma(N), \mathbbm{k}\big)\] that maps $\mathcal{E}_{k}\big(\Gamma(N), \mathbbm{k}\big)$ onto $\mathcal{E}_{k,\text{un}}\big(\Gamma(N), \mathbbm{k}\big)$.

Let $k$ be an integer $\geq 2$ and $\mathcal{G}$ be a congruence subgroup containing $\Gamma(N)$. We define the unnormalized spaces of Eisenstein series on $\mathcal{G}$ as 
\begin{equation*}
\mathcal{E}_{k,\text{un}}^{\ast}\big(\mathcal{G}, \mathbbm{k}\big) = \mathcal{E}_{k,\text{un}}^{\ast}\big(\Gamma(N), \mathbbm{k}\big)^{\mathcal{G}}, \hspace{.3cm} \mathcal{E}_{k,\text{un}}\big(\mathcal{G}, \mathbbm{k}\big) = \mathcal{E}^{\ast}_{k,\text{un}}\big(\mathcal{G}, \mathbbm{k}\big) \cap \text{Hol}(\mathbb{H}).  	
\end{equation*}
In this case, the spaces of Eisenstein series on $\mathcal{G}$ may depend on the choice of the underlying principal level. Note that if $\mathbbm{k} = \mathbb{C}$, then the unnormalized spaces equal their spectral version thanks to the corresponding equality for $\Gamma(N)$. One restricts $\Psi_{N, \mathcal{A}}$ to the $\mathcal{G}$-fixed subspaces of its domain and codomain to obtain an isomorphism \[\Psi_{\mathcal{G}, \mathcal{A}}: \mathcal{E}^{\ast}_{k}(\mathcal{G}, \mathbbm{k}) \xrightarrow{\cong} \mathcal{E}_{k,\text{un}}^{\ast}(\mathcal{G}, \mathbbm{k})\] that maps $\mathcal{E}_{k}(\mathcal{G}, \mathbbm{k})$ isomorphically onto $\mathcal{E}_{k,\text{un}}(\mathcal{G}, \mathbbm{k})$. If $k$ is odd, then assume that $-\text{Id} \notin \mathcal{G}$. We push forward the basis $\{E_{k,x} \mid \text{$x$ regular if $k$ odd}\}$ along $\Psi_{\mathcal{G}, \mathcal{A}}$ to discover that  
$\{G_{k,x} \mid \text{$x$ regular if $k$ odd}\}$ is a basis of $\mathcal{E}_{k,\text{un}}^{\ast}(\mathcal{G}, \mathbbm{k})$. Moreover, one can write down a basis for the space of Eisenstein series $\mathcal{E}_{k, \text{un}}(\mathcal{G}, \mathbbm{k})$ \`a la Theorem~\ref{theorem4.5new}. Letting $\mathbbm{k} = \mathbb{C}$, we thus obtain a basis for $\mathcal{E}_{k}(\mathcal{G})$ arising from the unnormalized Eisenstein series. 

\subsection{Algebraicity of Fourier coefficients}
\label{section5.2}
We begin by comparing the two rational structures for the space of Eisenstein series stemming from the spectral and unnormalized bases. The following elementary identity plays a crucial role in the later arguments. 

\begin{lemma}
\label{lemma5.3}\emph{\cite[p.60]{shimura}} 
Let $k \geq 2$ and $m \in \mathbb{Z}$. Then 
\[\sum_{\substack{n \in \mathbb{Z} - \{0\},\\ n \equiv m (N)}} \frac{1}{n^k} = - \frac{(2\pi i)^k}{k!N}\sum_{j = 0}^{N-1}\mathbf{e}(-\frac{jm}{N})B_{k}(\frac{j}{N})\]
where $B_{\bullet}(-)$ is the Bernoulli polynomial defined by $\frac{X\exp(tX)}{\exp(X) - 1} = \sum_{j \geq 0} \frac{B_j(t)}{j!}X^j$.  
\end{lemma}

\begin{proof}
Follows from the Fourier expansion formula for the Bernoulli polynomials:
\[B_{k}(t) = -\frac{k!}{(2\pi i)^k} \sum_{n \in \mathbb{Z} - \{0\}} \frac{\mathbf{e}(nt)}{n^k}; \hspace{.5cm}(\substack{k \geq 2,\; 0 \leq t \leq 1})\]
see Theorem~4.6 in \cite{shimura}. 
\end{proof}	

\begin{proposition}
\label{proposition5.4}
Let $\mathcal{G}$ be a congruence subgroup containing $\Gamma(N)$. Assume that $\mathbbm{k}$ is a subfield of $\mathbb{C}$ containing $\mathbb{Q}(\zeta_N)$. Then \[\text{$\mathcal{E}^{*}_{k,\emph{un}}(\mathcal{G}, \mathbbm{k}) = (2\pi i)^k\mathcal{E}^{\ast}_{k}(\mathcal{G}, \mathbbm{k})$ and $\mathcal{E}_{k,\emph{un}}(\mathcal{G}, \mathbbm{k}) = (2\pi i)^k\mathcal{E}_{k}(\mathcal{G}, \mathbbm{k})$.}\] 
\end{proposition} 

\begin{proof}
It is sufficient to verify that 
\begin{equation*}
\mathcal{E}^{*}_{k,\text{un}}\big(\Gamma(N), \mathbbm{k}\big) = (2\pi i)^k\mathcal{E}^{\ast}_{k}\big(\Gamma(N), \mathbbm{k}\big).	
\end{equation*}

If $k$ is odd and $N \in \{1,2\}$, then both sides of the identity are zero. Therefore, one can assume that $N \geq 3$ whenever $k$ is odd. Suppose that $N \in \{1,2\}$ and $k$ is even. In this situation, we employ \eqref{2.4new} to discover that \[G_{k}\big(\bar{\lambda},N\big) = 2\zeta_{+}^{(N)}(1, k)E_{k}(\bar{\lambda},N)\] for each $\bar{\lambda} \in \Lambda_N$. Since $\frac {\zeta(k)} {(2\pi i)^k} \in \mathbb{Q}^{\times}$, the required equality holds under the current hypothesis. Finally, let $N \geq 3$ and $\bar{\lambda} \in \Lambda_{N}$. Here one can rewrite \eqref{2.4new} as  
\begin{equation*}
G_{k}(\bar{\lambda},N) = \sum_{\substack{1 \leq \delta < \frac{N}{2},\\ \gcd(\delta, N) = 1}} \big(\zeta_{+}^{(N)}(\delta,k) + (-1)^k\zeta_{+}^{(N)}(N-\delta,k)\big)E_{k}(\bar{\delta}^{-1}
\bar{\lambda},N).	
\end{equation*}
Lemma~\ref{lemma5.3} implies that each coefficient in the RHS lies in $(2\pi i)^k\mathbbm{k}$. Therefore $G_{k}(\bar{\lambda}, N) \in (2\pi i)^{k} \mathcal{E}^{\ast}_{k}\big(\Gamma(N), \mathbbm{k}\big)$. As a consequence $\mathcal{E}^{*}_{k,\text{un}}\big(\Gamma(N), \mathbbm{k}\big) \subseteq (2\pi i)^{k} \mathcal{E}^{\ast}_{k}\big(\Gamma(N), \mathbbm{k}\big)$. But $\mathcal{E}^{*}_{k}\big(\Gamma(N), \mathbbm{k}\big)$ and $\mathcal{E}^{\ast}_{k,\text{un}}\big(\Gamma(N), \mathbbm{k}\big)$ have the same dimension over $\mathbbm{k}$. Hence, the inclusion above must be equality.     
\end{proof}

\begin{remark}
\label{remark5.4new}
Let $\delta$ and $N$ be two positive integers $\delta \mid N$. A direct calculation using the defining series \eqref{2.2new} shows that $\mathcal{E}_{k}^{\ast}\big(\Gamma(\delta),\mathbbm{k}\big) \subseteq \mathcal{E}^{*}_{k}\big(\Gamma(N),\mathbbm{k}\big)$. But this inclusion is not straight-forward to verify for the spaces spanned by unnormalized series. If $\mathbbm{k}$ is a subfield of $\mathbb{C}$ containing $\mathbb{Q}(\zeta_N)$, then Proposition~\ref{proposition5.4} implies $\mathcal{E}_{k, \text{un}}^{\ast}\big(\Gamma(\delta),\mathbbm{k}\big) \subseteq \mathcal{E}^{*}_{k, \text{un}}\big(\Gamma(N),\mathbbm{k}\big)$ also.
\end{remark}

Equipped with the necessary prerequisites, we proceed to prove the algebraicity of the Fourier coefficients theorem from the introduction.

\begin{theorem}
\label{theorem5.5new}
Let $\mathbbm{k}$ be a subfield of $\mathbb{C}$ containing $\mathbb{Q}(\zeta_N)$. Then, each element of $\mathcal{E}_{k}(\mathcal{G},\mathbbm{k})$ admits a Fourier expansion of the form \[\text{$\sum_{j \geq 0} A_j\mathbf{e}(\frac{j\tau}{N})$ with $A_j \in \mathbbm{k}$.}\] In particular, each element of the spectral basis $\mathcal{B}_{\mathcal{G},k}$ has the Fourier coefficients in $\mathbb{Q}(\zeta_N)$.    
\end{theorem}

\begin{proof}
The Fourier expansion formula \eqref{2.6new} shows that \[(2\pi i)^{-k}G_{k}(\bar{\lambda}, N) = \sum_{j \geq 0} A_j \mathbf{e}(\frac{j\tau}{N}) + (2\pi i)^{-k}P_k(\tau)\] for each $\bar{\lambda} \in \Lambda_{N}$ where $A_j$-s lie in $\mathbb{Q}(\zeta_N)$. For each $\mathcal{G}$-orbit $\mathcal{O}$ in $\Lambda_{N}$, one can construct unnormalized orbital sum $G_{k}(\mathcal{O}, \mathcal{G}; s)$ exactly as in Proposition~\ref{proposition5.1}. Now the assertion above implies that $\mathcal{E}^{\ast}_{k}(\mathcal{G},\mathbbm{k}) = (2\pi i)^{-k}\mathcal{E}_{k, \text{un}}^{\ast}(\mathcal{G},\mathbbm{k})$ admits a basis so the Fourier coefficients of the holomorphic part of each basis function lie in $\mathbb{Q}(\zeta_N)$. Proof of the first part is now clear. The second part follows from the first part since $\mathcal{B}_{\mathcal{G},k}$ is contained $\mathcal{E}_k\big(\mathcal{G}, \mathbb{Q}(\zeta_N)\big)$. 
\end{proof}

\section{Application to the Eichler-Shimura theory}   
\label{section6} 
Let $k$ be an integer $\geq 2$ and set $V_{k-2,\mathbb{C}} = \{P \in \mathbb{C}[X] \mid \text{deg}(P) \leq k-2\}$. There is a right $\mathbb{C}[\Gamma]$-module stricture on $V_{k-2,\mathbb{C}}$ described by 
\[P(X)\lvert_{\gamma} = (cX+d)^{k-2}P(\frac{aX+b}{cX+d}). \hspace{.3cm}(\substack{\gamma \in \Gamma})\]
Suppose that $\mathcal{G}$ is a congruence subgroup and $f \in \mathcal{M}_k(\mathcal{G})$. With each $\tau_0 \in \mathbb{H}$ one associates a cocyle as follows:  
\[\mathsf{p}_{k}^{\tau_0}(f) : \mathcal{G} \to V_{k-2, \mathbb{C}}, \hspace{.3cm} \gamma \mapsto \int^{\tau_0}_{\gamma^{-1}\tau_0}f(\tau)(X-\tau)^{k-2}d\tau.\]
The image of $\mathsf{p}_{k}^{\tau_0}(f)$ in $H^{1}(\mathcal{G}, V_{k-2, \mathbb{C}})$ depends only $f$ and is independent of the choice of $\tau_0$. We define the Eichler-Shimura homomorphism 
\begin{equation}
\label{6.1new}
[\text{ES}_k]: \mathcal{M}_{k}(\mathcal{G}) \to H^{1}(\mathcal{G}, V_{k-2,\mathbb{C}})	
\end{equation}
by setting $[\text{ES}_k](f) = [\mathsf{p}_{k}^{\tau_0}(f)]$ for some $\tau_0 \in \mathbb{H}$. Now suppose $\mathcal{S}_k^{a}(\mathcal{G}) = \{\bar{f} \mid f \in \mathcal{S}_k(\mathcal{G})\}$ is the space of anti-holomorphic cusp forms. We extend the Eichler-Shimura homomorphism \eqref{6.1new} to the space of antiholomorphic cusp forms by complex conjugation: 
\[[\text{ES}_k]^{a}: \mathcal{S}_k^{a}(\mathcal{G}) \to H^{1}(\mathcal{G}, V_{k-2,\mathbb{C}}), \hspace{.3cm} f \mapsto \#[\text{ES}_k](\bar{f});\] where $\#$ is complex conjugation on the cohomology arising from the conjugation on the coefficients.
The Eichler-Shimura homomorphisms together give rise to a Hecke equivariant $\mathbb{C}$-linear map  
\begin{equation}
\label{6.2new}
[\text{ES}_k] \oplus [\text{ES}_k]^{a}: \mathcal{M}_{k}(\mathcal{G}) \oplus \mathcal{S}_{k}^{a}(\mathcal{G}) \xrightarrow{} H^{1}(\mathcal{G}, V_{k-2,\mathbb{C}}). 
\end{equation}
The classical Eichler-Shimura theorem for cusp forms \cite[6.2]{hida} asserts that the homomorphism above is injective on $\mathcal{S}_{k}(\mathcal{G}) \oplus \mathcal{S}_{k}^{a}(\mathcal{G})$ and sends this subspace of the domain onto the parabolic subspace, denoted $H^{1}_{c}(\mathcal{G}, V_{k-2,\mathbb{C}})$, of the group cohomology. We extend the Eichler-Shimura isomorphism theorem for cusp forms to the whole space of modular forms as promised in the introduction. 

\begin{theorem}
\label{theorem6.1new}
The Eichler-Shimura homomorphism \eqref{6.2new} is a Hecke equivariant isomorphism. 
\end{theorem} 

Recall that the construction of parabolic subspace \cite[6]{hida} yields an exact sequence: 
\[0 \to H^{1}_{c}(\mathcal{G}, V_{k-2, \mathbb{C}}) \to H^{1}(\mathcal{G}, V_{k-2,\mathbb{C}}) \to \oplus_{x \in \mathcal{C}(\mathcal{G})} H^{1}(\mathcal{G}_x, V_{k-2, \mathbb{C}}).\]
Therefore, there is a commutative diagram:
\begin{equation}
\label{6.3new}
\begin{tikzcd}[column sep=small]
\mathcal{S}_k(\mathcal{G}) \oplus \mathcal{S}_k^{a}(\mathcal{G}) \arrow[d, "\cong"', "{[\text{ES}_k] \oplus [\text{ES}_k]^{a}}"] \arrow[r, hookrightarrow] &  \mathcal{M}_k(\mathcal{G}) \oplus \mathcal{S}_k^{a}(\mathcal{G})  \arrow[d, "{[\text{ES}_k] \oplus [\text{ES}_k]^{a}}"] \arrow[r, twoheadrightarrow] & \mathcal{E}_k(\mathcal{G}) \arrow[d, "\Theta"] \\
H^{1}_{c}(\mathcal{G}, V_{k-2,\mathbb{C}})\arrow[r, hookrightarrow] & H^{1}(\mathcal{G}, V_{k-2,\mathbb{C}}) \arrow[r] & \oplus_{x \in \mathcal{C}(\mathcal{G})} H^{1}\big(\mathcal{G}_x, V_{k-2,\mathbb{C}}\big) 
\end{tikzcd}
\end{equation}
where $\Theta = (\Theta_x)_{x \in \mathcal{C}(\mathcal{G})}$ is the natural map induced by quotient. We begin by recalling some group cohomological preliminaries necessary for our arguments. Let $\mathcal{H}$ be an abstract group and $V$ be a right $\mathbb{C}[\mathcal{H}]$-module. As usual, we denote the collection of cocycles by $Z^{1}(\cdot, \cdot)$. Suppose that $\mathcal{G}$ is a subgroup of $\mathcal{H}$. Given $h \in \mathcal{H}$, there is an isomorphism at the level of cocycles
\[Z^{1}(\mathcal{G}, V) \to Z^{1}(h^{-1}\mathcal{G}h, V); \hspace{.3cm} \text{$P \mapsto P|_{h}$, $P\lvert_{h}(x) := P(hxh^{-1})h$}\]
which gives rise to an isomorphism $H^{1}(\mathcal{G}, V) \cong H^{1}(h^{-1}\mathcal{G}h, V)$. Suppose that $\mathcal{G}$ is a normal subgroup of $\mathcal{H}$. Then $[P] \mapsto [P\lvert_{h}]$ is a well-defined $\mathcal{H}$ action on $H^{1}(\mathcal{G},V)$. In this situation the action of $\mathcal{H}$ on $H^{1}(\mathcal{G}, V)$ factors through $\mathcal{G}$ and restriction yields an isomorphism $H^{1}(\mathcal{H}, V) \cong H^{1}(\mathcal{G}, V)^{\mathcal{H}}$.

\begin{remark}
\label{remark6.2new}
Let the notation be as above but we do not assume $\mathcal{G}$ to be a normal subgroup of $\mathcal{H}$. By general theory \[\text{cores}^{\mathcal{H}}_{\mathcal{G}} \circ \text{res}^{\mathcal{H}}_{\mathcal{G}} = \text{multiplication by $[\mathcal{H}: \mathcal{G}]$}\] where $\text{res}^{\mathcal{H}}_{\mathcal{G}}: H^{*}(\mathcal{H}, V) \to H^{*}(\mathcal{G}, V)$ is the restriction map and $\text{cores}^{\mathcal{H}}_{\mathcal{G}}: H^{*}(\mathcal{G}, V) \to H^{*}(\mathcal{H}, V)$ is the corestriction map \cite[III.9]{brown}. In particular, the restriction map is injective.
\end{remark}

Our first job is to analyze the boundary cohomology groups arising in the bottom right corner of \eqref{6.3new}. For a subset $\mathcal{S}$ of $\Gamma$, let $\langle \mathcal{S} \rangle$ denote the subgroup generated by $\mathcal{S}$. Note that 
\begin{equation}
\label{6.4new}
H^{1}(\langle T^h \rangle, V_{k-2,\mathbb{C}}) \cong \frac{V_{k-2,\mathbb{C}}}{V_{k-2,\mathbb{C}}(T^h-1)} \cong \mathbb{C}
\end{equation}
where the second isomorphism stems from the linear functional $V_{k-2, \mathbb{C}} \to \mathbb{C}$ described by $P(X) \mapsto$ the coefficient of $X^{k-2}$ in $P$. Suppose that $x \in \mathcal{C}(\mathcal{G})$ and choose $\gamma \in \Gamma$ so that $x = \mathcal{G}(\gamma \infty)$. Then $\gamma^{-1}\mathcal{G}_x\gamma = \gamma^{-1}\mathcal{G}\gamma \cap \Gamma_{\infty}$. In more detail, there exists a positive integer $h$ so that  
\begin{equation*}
	\gamma^{-1}\mathcal{G}_x\gamma = \begin{cases}
		\langle -\text{Id}, T^h \rangle \text{ or } \langle T^h \rangle, & \text{if $x \in \mathcal{C}_{\infty}(\mathcal{G})$;}\\
		\langle -T^h \rangle, & \text{if $x \notin \mathcal{C}_{\infty}(\mathcal{G})$}
	\end{cases}
\end{equation*}
where $\mathcal{C}_{\infty}(\mathcal{G})$ is the collection of regular cusps. The following result is certainly well-known among the experts; cf. \cite[6.3]{hida}. 

\begin{proposition}
\label{proposition6.3new} 
Let $\mathcal{G}$ be a congruence subgroup and assume that $-\emph{Id} \notin \mathcal{G}$ if $k$ is odd. Then 
\[H^{1}(\mathcal{G}_x, V_{k-2,\mathbb{C}}) \cong \begin{cases}
		\mathbb{C}, & \text{if $k$ is even;}\\
		\mathbb{C}, & \text{if $k$ is odd and $x \in \mathcal{C}_{\infty}(\mathcal{G})$;}\\
		0, & \text{if $k$ is odd and $x \notin \mathcal{C}_{\infty}(\mathcal{G})$.}
	\end{cases}\]
\end{proposition}
\begin{proof}
We use the conjugation by $\gamma$ isomorphism above to discover \[H^{1}(\mathcal{G}_x, V_{k-2,\mathbb{C}}) \cong H^{1}(\gamma^{-1}\mathcal{G}_x\gamma, V_{k-2,\mathbb{C}}).\] If $\gamma^{-1}\mathcal{G}_x\gamma = \langle T^h \rangle$ then $H^{1}(\gamma^{-1}\mathcal{G}_x\gamma, V_{k-2,\mathbb{C}}) \cong \mathbb{C}$. Now suppose $\gamma^{-1}\mathcal{G}_x\gamma = \langle -\text{Id}, T^h \rangle$.  Observe that 
\[H^{1}(\langle -\text{Id}, T^h \rangle, V_{k-2,\mathbb{C}}) \cong H^{1}(\langle T^h \rangle, V_{k-2,\mathbb{C}})^{\langle -\text{Id}, T^h \rangle}.\]   
Here $k$ is even and the element $-\text{Id}$ acts trivially on $H^{1}(\langle T^h \rangle, V_{k-2,\mathbb{C}})$. As a consequence $H^{1}(\langle -\text{Id}, T^h \rangle, V_{k-2,\mathbb{C}}) \cong \mathbb{C}$. It remains to consider the case $\gamma^{-1}\mathcal{G}_x\gamma = \langle -T^h \rangle$. In this situation, there is an isomorphism  
\[H^{1}(\langle -T^{h} \rangle, V_{k-2,\mathbb{C}}) \cong H^{1}(\langle T^{2h} \rangle, V_{k-2,\mathbb{C}})^{\langle -T^{h}\rangle}.\]
But $-T^{h}$ acts on $H^{1}(\langle T^{2h} \rangle, V_{k-2,\mathbb{C}})$ by $(-1)^{k-2}$. Hence $H^{1}(\langle -T^{h} \rangle, V_{k-2,\mathbb{C}}) \cong \mathbb{C}$ if $k$ is even and $0$ if $k$ is odd.
\end{proof}
The next lemma is also familiar in literature, but it's often difficult to find a concise proof.  

\begin{lemma}
\label{lemma6.4new}
Let $f \in \mathcal{M}_{k}(\mathcal{G})$ and $x \in \mathcal{C}(\mathcal{G})$. We choose $\gamma \in \Gamma$ with $\gamma\infty = x$. Then the restriction of $[\emph{ES}_k](f)$ to $H^{1}(\mathcal{G}_x, V_{k-2,\mathbb{C}})$ is zero if and only if $\pi_{\infty}(f\lvert_{ \gamma}) = 0$.  
\end{lemma}

\begin{proof}
Let $\Gamma(N)$ be a principal congruence subgroup contained in $\mathcal{G}$ with $N \geq 3$. The compatibility of the Eichler-Shimura cocyles for $\mathcal{G}$ and $\Gamma(N)$ yields the following commutative diagram: 
\begin{equation*}
\begin{tikzcd}
\mathcal{M}_k(\mathcal{G}) \arrow[d]\arrow[r, "{[\text{ES}_k]}"] & H^{1}(\mathcal{G}_x, V_{k-2,\mathbb{C}})\arrow[d]\\
\mathcal{M}_{k}\big(\Gamma(N)\big) \arrow[r, "{[\text{ES}_k]}"] & H^{1}\big(\Gamma(N)_x, V_{k-2,\mathbb{C}}\big) 
\end{tikzcd}
\end{equation*}
where the left vertical arrow is an inclusion and the right vertical arrow is a 
restriction map. Now Remark~\ref{remark6.2new} shows that the right arrow is injective. Thus, it suffices to verify the statement for $\mathcal{G} = \Gamma(N)$. The conjugation action of $\gamma$ on $H^{1}\big(\Gamma(N), V_{k-2,\mathbb{C}}\big)$ induces a commutative diagram
\begin{equation*}
\begin{tikzcd}
H^{1}\big(\Gamma(N), V_{k-2, \mathbb{C}}\big) \arrow[d]\arrow[r, "{[P] \mapsto [P\lvert_{ \gamma}]}", "\cong"'] & H^{1}\big(\Gamma(N), V_{k-2, \mathbb{C}}\big) \arrow[d]\\
H^{1}\big(\Gamma(N)_{\gamma \infty}, V_{k-2, \mathbb{C}}\big) \arrow[r, "{[P] \mapsto [P\lvert_{ \gamma}]}", "\cong"'] & H^{1}\big(\Gamma(N)_{\infty}, V_{k-2, \mathbb{C}}\big). 
\end{tikzcd}
\end{equation*}
Here the horizontal maps are isomorphisms and the vertical arrows are restriction maps. Now, a direct calculation shows that $[\text{ES}_k](f)\lvert_{ \gamma} = [\text{ES}_k](f\lvert_{ \gamma})$. Write 
$F(\tau) = f\lvert_{\gamma}(\tau) - \pi_{\infty}(f\lvert_{\gamma})$. Then $F$ exhibits exponential decay in a neighborhood of $\infty$. We have 
\[\mathsf{p}_{k}^{\tau_0}(f\lvert_{ \gamma})(T^N) = \pi_{\infty}(f\lvert_{ \gamma})\int^{\tau_0}_{\tau_0 -N}(X-\tau)^{k-2}d\tau + \int^{\tau_0}_{\tau_0 -N}F(\tau)(X-\tau)^{k-2}d\tau.\] 
The coefficient of $X^{k-2}$ in LHS does not depend on the choice of $\tau_0$. The coefficient of $X^{k-2}$ in the first term of RHS equals $N\pi_{\infty}(f\lvert_{ \gamma})$. Letting $\tau_0 \to \infty$ along the imaginary axis we see that the image of $\mathsf{p}_k^{\tau_0}(f\lvert_{ \gamma})$ in $H^{1}\big(\Gamma(N)_{\infty}, V_{k-2,\mathbb{C}}\big) \cong \mathbb{C}$ equals $N \pi_{\infty}(f\vert_{\gamma})$. Hence the restriction of $[\text{ES}_k](f)$ to $H^{1}\big(\Gamma(N)_{\gamma \infty}, V_{k-2,\mathbb{C}}\big)$ vanishes if and only if $\pi_{\infty}(f\lvert_{ \gamma}) = 0$. 
\end{proof}

The action of $\Gamma$ on $\mathbb{H}$ factors through $\mathsf{P}\Gamma = \Gamma/\{\pm\text{Id}\}$, that is to say, $-\text{Id}$ acts trivially on each point. Suppose that $\mathsf{P}\mathcal{G}$ is the image of $\mathcal{G}$ in $\mathsf{P}\Gamma$. If $k$ is even, then the $\mathcal{G}$-action on $V_{k-2,\mathbb{C}}$ descends to $\mathsf{P}\mathcal{G}$ and the quotient map induces an isomorphism $H^{1}(\mathcal{G}, V_{k-2,\mathbb{C}}) \cong H^{1}(\mathsf{P}\mathcal{G}, V_{k-2,\mathbb{C}})$ that restricts to an isomorphism between the parabolic subspaces. For $k$ odd, we additionally assume that $-\text{Id} \notin \mathcal{G}$. Then $V_{k-2,\mathbb{C}}$ acquires a $\mathsf{P}\mathcal{G}$-module structure via the isomorphism $\mathcal{G} \cong \mathsf{P}\mathcal{G}$. For $x \in \mathcal{C}(\mathcal{G})$, let $\pi_x$ denote a generator for the infinite cyclic group $\mathsf{P}\mathcal{G}_x$. We also fix a set of representatives $\{\varepsilon_1, \ldots, \varepsilon_r\}$ for the elliptic elements in $\mathsf{P}\mathcal{G}$. The following dimension formula is necessary to handle the extra difficulty for weight $2$: 

\begin{proposition}
\label{proposition6.5new}
Let $V$ be a finite-dimensional $\mathbb{C}[\mathsf{P}\mathcal{G}]$-module and $g$ denote the genus of the compact Riemann surface attached to $\mathsf{P}\mathcal{G} \backslash \mathbb{H}$. Then
\begin{gather*}
\sum_{j=0}^{2} (-1)^j \emph{dim}H^{j}(\mathsf{P}\mathcal{G}, V) = \big(2- 2g - \lvert \mathcal{C}(\mathcal{G})\rvert\big) \emph{dim} V - \sum_{j=1}^{r}(\emph{dim} V - n_j^{\emph{ell}}),\\  
\emph{dim}H^{1}_{c}(\mathsf{P}\mathcal{G}, V) = (2g-2)\emph{dim} V + n_0 + n_1 + \sum_{x \in \mathcal{C}(\mathcal{G})} n_x^{\emph{par}} + \sum_{j=1}^{r} (\emph{dim} V - n^{\emph{ell}}_j)  
\end{gather*}
where $n_0 = \emph{dim}V^{\mathsf{P}\mathcal{G}}$, $n_1 = \emph{dim}H^{2}_c(\mathsf{P}\mathcal{G}, V)$, $n_j^{\emph{ell}} = \emph{dim} \{v \in V \mid v\varepsilon_j = v\}$, and $n_x^{\emph{par}} = \emph{dim} \{v(\pi_x -1) \mid v \in V\}$. 
\end{proposition}
\begin{proof}
The second identity is standard, and we refer the readers to \cite[p. 229]{shimura94} for a demonstration. The first identity is a consequence of the arguments employed to establish the other identity.
\end{proof}

We combine the observations above to arrive at the proof of Theorem~\ref{theorem6.1new}. 

\paragraph{Proof of Theorem~\ref{theorem6.1new}.} First, observe that both the domain and codomain of \eqref{6.2new} are zero if $k$ is odd and $-\text{Id} \in \mathcal{G}$. Assume that $-\text{Id} \notin \mathcal{G}$ whenever $k$ is odd. Now suppose $k \geq 3$. Write $\mathcal{C}_k(\mathcal{G}) := \mathcal{C}(\mathcal{G})$ if $k$ is even, and $\mathcal{C}_{\infty}(\mathcal{G})$ if $k$ is odd. Let $x, y \in \mathcal{C}_{k}(\mathcal{G})$. Lemma~\ref{lemma4.2new} and Lemma~\ref{lemma6.4new} together imply that $\Theta_y(E_{k,x}) \neq 0$ if and only if $x = y$ in $\mathcal{C}_k(\mathcal{G})$. Now, an application of Proposition~\ref{proposition6.3new} yields, $\Theta$ is injective and surjective onto $\oplus_{x \in \mathcal{C}_k(\mathcal{G})} H^{1}(\mathcal{G}_x, V_{k-2,\mathbb{C}})$. It follows that the middle arrow of \eqref{6.3new} is an isomorphism. 

Now suppose $k = 2$ and let $x, y \in \mathcal{C}(\mathcal{G}) - \{x_0\}$. In this case, we have $\Theta_y\big(E_{2,x} - \frac{\lvert \text{or}_{\mathcal{G},N}(x)\rvert}{\lvert \text{or}_{\mathcal{G},N}(x_0)\rvert}E_{2,x_0}\big) \neq 0$ if and only if $x = y$ in $\mathcal{C}(\mathcal{G})$. Therefore, $\Theta$ is an injective map. To finish off the proof it suffices to verify that 
\begin{equation}
\label{6.5new}
H^{1}(\mathcal{G}, \mathbb{C}) = H^{1}_c(\mathcal{G}, \mathbb{C}) + \text{dim} \mathcal{E}_k(\mathcal{G}) = H^{1}_c(\mathcal{G}, \mathbb{C}) + \lvert \mathcal{C}(\mathcal{G}) \rvert -1. 
\end{equation}
One uses the identities in Proposition~\ref{proposition6.3new} for this purpose. The formula for $H^{2}(\mathsf{P}\mathcal{G}, \cdot)$ given in Proposition~8.2 of \cite[8.1]{shimura94} implies $n_1 = 1$. Now suppose $\mathsf{P}\mathcal{G}'$ is a torsion-free finite-index subgroup of $\mathsf{P}\mathcal{G}$. Since open surfaces have no cohomology in degree $2$, it follows that $H^{2}(\mathsf{P}\mathcal{G}', \mathbb{C})$. By Remark~\ref{remark6.2new}, the restriction map $H^{2}(\mathsf{P}\mathcal{G}, \mathbb{C}) \to H^{2}(\mathsf{P}\mathcal{G}', \mathbb{C})$ is injective. Thus $H^{2}(\mathsf{P}\mathcal{G}, \mathbb{C})$ is also zero. Now the identity \eqref{6.5new} follows from Proposition~\ref{proposition6.3new}. \hfill $\square$

\section{Connections with the arithmetic theory}
\label{section7new}
\subsection{Eisenstein series with nebentypus}
\label{section7.1new}
We begin by introducing a few notation useful for the current discussion. For a finite abelian group $A$, let $\widehat{A}$ denote the dual group described by $\widehat{A} := \text{Hom}(A, \mathbb{C}^{\times})$. If $A \twoheadrightarrow B$ is a quotient map with kernel $K$ then one identifies $\widehat{B}$ with the annihilator of $K$ in $\widehat{A}$. Put 
\[U_N = (\mathbb{Z}/N)^{\times}. \hspace{.3cm}(\substack{N \geq 1})\]
If $M \mid N$ then the surjection $U_N \twoheadrightarrow U_M$ turns $\widehat{U}_M$ into a subgroup of $\widehat{U}_N$. 

Let $k$ be an integer $\geq 2$ and $\chi \in \widehat{U}_N$. We consider $\chi$ as a character on $\Gamma_{0}(N)$ through the canonical homomorphism $\Gamma_0(N) \twoheadrightarrow U_N$ described by 
$\begin{pmatrix}
	a & b\\
	c & d
\end{pmatrix} \mapsto d (\text{mod }N)$. Set \[\mathcal{M}_k(N, \chi) = \{f \in \mathcal{M}_k\big(\Gamma_{1}(N)\big) \mid f\lvert_{ \gamma} = \chi(\gamma)f, \forall \gamma \in \Gamma_{0}(N)\}\] 
and write \[\text{$\mathcal{E}_k(N, \chi) = \mathcal{M}_k(N, \chi) \cap \mathcal{E}_k\big(\Gamma_{1}(N)\big)$, $\mathcal{S}_k(N, \chi) = \mathcal{M}_k(N, \chi) \cap \mathcal{S}_k\big(\Gamma_{1}(N)\big)$.}\] 
Note that $\mathcal{M}_k(N, \chi) = \{0\}$ unless $\chi(-1) = (-1)^{k}$. Moreover, there are direct sum decompositions 
\begin{equation*}
	\begin{aligned}
		\mathcal{M}_k(N, \chi) & = \mathcal{E}_k(N, \chi) \oplus \mathcal{S}_k(N, \chi),\\
		W_k\big(\Gamma_{1}(N)\big) & = \text{$\bigoplus_{\chi \in \widehat{U}_N} W_k(N, \chi)$ where $W$ is one of $\{\mathcal{M}, \mathcal{S}, \mathcal{E}\}$.}
	\end{aligned}
\end{equation*}
We call the elements of $\mathcal{M}_k(N, \chi)$ modular forms of weight $k$ with nebentypus $\chi$. This subsection primarily aims to describe a basis for the space of Eisenstein series $\mathcal{E}_{k}(N, \chi)$ parameterized by the cusps of $\Gamma_{0}(N)$ as explained in the introduction. If $N \in \{1,2\}$ then the eigenspace decomposition of the space of Eisenstein series has only one summand, namely $\mathcal{E}_{k}\big(\Gamma_{1}(N)\big)$, which already admits a spectral basis by the theory in Section~\ref{section4}. Thus without loss of generality, we can assume that $N \geq 3$.  Our discussion freely uses the description of the spectral bases for $\Gamma_{1}(N)$ and $\Gamma_{0}(N)$ given in Appendix~\ref{appendix}.    

Let $N$ be a positive integer $\geq 3$ and $\delta$ be a positive divisor of $N$. A character $\chi \in \widehat{U}_N$ is \textit{$\delta$-good} if there exists a factorization $\chi = \chi_1 \chi_2$ with $\chi_1 \in \widehat{U}_{\frac{N}{\delta}}$ and $\chi_2 \in \widehat{U}_{\delta}$. Given an orbit of $\Gamma_{0}(N)$ in $\Lambda_{N}$ (Appendix~\ref{appendixA.2}) described by \[\lambda = (\lambda_1, \bullet) \in \mathcal{O}_{N}^{(0)}\] our theory attaches an Eisenstein series with the cusp described by the orbit of $\lambda$ only if $\chi$ is $\delta$-good where $\delta = \gcd(\lambda_1, N)$; cf. \cite[3.1]{young}. Observe that for $\lambda = (0,1)$ the latter condition imposes no constraint on $\chi$. If $\lambda = (\frac{N}{2}, 1)$ and $\chi$ is $\frac{N}{2}$-good then there is a unique factorization $\chi = \chi_1 \chi_2$ with $\chi_1 = \text{trivial}$ and $\chi_2 = \chi$. Let $k$ be a positive integer $\geq 2$ and $s$ be a complex number with $\text{Re}(k+2s) > 1$. Recall the notation $E_k^{(1)}(\cdot)$ from \eqref{A.1}. Define 
\begin{equation}
\label{7.1new}
\begin{gathered}
E_{k, \chi}^{(0)}\big((0,1), N; s\big) := \frac{1}{2}\sum_{\substack{0 \leq \lambda_2 \leq N-1,\\ \gcd(\lambda_2, N) = 1}} \chi(\lambda_2)^{-1} E_{k}^{(1)}\big((0,\lambda_2), N;s\big),\\
E_{k, \chi}^{(0)}\big((\frac{N}{2},1), N; s\big) \hspace{8cm} \\
:= \frac{1}{2}\sum_{\substack{0 \leq \lambda_2 \leq \frac{N}{2}-1,\\ \gcd(\lambda_2, \frac{N}{2}) = 1}} \chi(\lambda_2)^{-1} E_{k}^{(1)}\big((\frac{N}{2},\lambda_2), N;s\big) \hspace{.2cm}\text{if $\chi \in \widehat{U}_{\frac{N}{2}}$.} 
\end{gathered}
\end{equation}

Let $(\delta, \lambda_0) \in \mathcal{O}_{N}^{(0)}$ where $1 \leq \delta < \frac{N}{2}$ and $\lambda_0$ is an integer coprime to $\gcd(\delta, \frac{N}{\delta})$ satisfying $0 \leq \lambda_0 \leq \gcd(\delta, \frac{N}{\delta}) - 1$. Suppose that $\chi$ is a $\delta$-good character and $\chi = \chi_1 \chi_2$ is a factorization of $\chi$. With this data, one attaches a series    
\begin{equation}
\label{7.2new}
E_{k, \chi_1, \chi_2}^{(0)}\big((\delta,\lambda_0), N; s\big) := \frac{1}{2}\sum_{(a, \lambda_2)} \chi_1(a)\chi_2(\lambda_2)^{-1} E_{k}^{(1)}\big((a\delta,\lambda_2),N;s\big)
\end{equation}
where the indexing set for the sum is the same as \eqref{A.3}. It is straightforward to check that the functions defined by \eqref{7.1new}-\eqref{7.2new} satisfy \[f(\tau;s)\lvert_{k,\gamma} = \chi(\gamma)f(\tau;s), \hspace{.3cm} \forall \gamma \in \Gamma_{0}(N).\] 
In particular these functions are zero unless $\chi(-1) = (-1)^k$. One can employ the correspondence between $\mathcal{O}_{N}^{(0)}$ and the cusps of $\Gamma_{0}(N)$ as in the discussion around \eqref{4.6new} to parameterize these Eisenstein series using the cusps of $\Gamma_{0}(N)$. But for convenience of presentation, we prefer the orbital sum picture in this article. The following lemma records a few preliminary observations regarding the formalism introduced above.

\begin{lemma}
\label{lemma7.1new}
Let $\chi \in \widehat{U}_N$ and $\delta$ be a positive divisor of $N$. Then 
\begin{enumerate}[label=(\roman*), align=left, leftmargin=0pt]
\item $\chi$ is $\delta$-good if and only if $\chi \in \widehat{U}_{\emph{lcm}(\delta, \frac{N}{\delta})}$. 
		
\item Let the notation be as in \eqref{7.2new}. Suppose that $\chi = \chi_1\chi_2 = \chi_1'\chi_2'$ are two factorizations attached to $\chi$. Then there exists a nonzero complex number $\alpha$ so that 
\[E_{k, \chi_1, \chi_2}^{(0)}\big((\delta,\lambda_0), N; s\big) = \alpha E_{k, \chi'_1, \chi'_2}^{(0)}\big((\delta,\lambda_0), N; s\big).\]
\end{enumerate}	
\end{lemma} 

\begin{proof}
\begin{enumerate}[label=(\roman*), align=left, leftmargin=0pt]
\item The subgroup generated by $\widehat{U}_{\frac{N}{\delta}}$ and $\widehat{U}_{\delta}$ equals $\widehat{U}_{\text{lcm}(\delta, \frac{N}{\delta})}$. Thus $\chi$ is $\delta$-good if and only if $\chi \in \widehat{U}_{\text{lcm}(\delta, \frac{N}{\delta})}$.     	
\item Write $\psi = \chi_1'\chi_1^{-1} = \chi_2\chi_2'^{-1}$. Then $\psi \in \widehat{U}_{\gcd(\delta, \frac{N}{\delta})}$. Now, an easy calculation using \eqref{7.2new} demonstrates that $\alpha = \psi(\lambda_0)^{-1}$ satisfies the required property. 
\end{enumerate}	
\end{proof}

For $\chi \in \widehat{U}_N$, set 
\[\mathcal{O}_{N,\chi}^{(0)} = \{\lambda = (\lambda_1, \bullet) \in \mathcal{O}_{N}^{(0)} \mid \text{$\chi$ is $\delta$-good where $\delta = \gcd(\lambda_1, N)$}\}.\]
If $(\delta, \lambda_0) \in \mathcal{O}_{N,\chi}^{(0)}$ for some $1 \leq \delta < \frac{N}{2}$, then we fix a factorization of $\chi = \chi_1 \chi_2$ and write $E_{k, \chi}^{(0)}$ instead of $E_{k, \chi_1, \chi_2}^{(0)}$. As before one abbreviates the specializations of the Eisenstein series at $s = 0$ by $E_{k, \chi}^{(0)}(\lambda, N)$. Observe that the function $E_{k, \chi}^{(0)}(\lambda, N)$ is holomorphic on $\mathbb{H}$ unless $k = 2$ and $\chi$ is trivial. Since the latter case is already covered by Theorem~\ref{theorem4.5new} we exclude it from our discussion below.

\begin{theorem}
\label{theorem7.2new}
Let $k$ be an integer $\geq 2$. Suppose that $N \geq 3$ and $\chi \in \widehat{U}_N$ with $\chi(-1) = (-1)^{k}$. If $k = 2$ then assume that $\chi$ is nontrivial. Then 
\[\mathcal{B}_{k,\chi} := \{E_{k, \chi}^{(0)}(\lambda,N) \mid \lambda \in \mathcal{O}_{N, \chi}^{(0)}\}\] 
is a basis for $\mathcal{E}_k(N, \chi)$.  
\end{theorem}

\begin{proof}
We begin by showing that $\mathcal{B}_{k,\chi}$ is a linearly independent subset. For each $\lambda \in \mathcal{O}^{(0)}_{N,\chi}$ choose $\gamma_{\lambda} \in \Gamma$ so that 
\begin{equation*}
\begin{cases}
(0, 1)\gamma_{\lambda} = (0,1), & \text{if $\lambda = (0,1)$;}\\
(\delta, \lambda_2)\gamma_{\lambda} = (0,1), & \text{if $\lambda = (\delta,\lambda_0);$}\\
(\frac{N} {2}, 1)\gamma_{\lambda} = (0,1), & \text{if $\lambda = (\frac{N}{2},1)$}
\end{cases}
\end{equation*}
holds in $\Lambda$. Here $\lambda_2$ is a fixed integer with $0 \leq \lambda_2 \leq \delta -1$, $\gcd(\delta, \lambda_2) = 1$, and $\lambda_2 \equiv \lambda_0 \big($mod $\gcd(\delta, \frac{N}{\delta})\big)$. Then 
\begin{equation*}
\text{$\pi_{\infty}\big(E_{k,\chi}^{(0)}(\lambda',N)\lvert_{\gamma_{\lambda}}\big) \neq 0$ if and only if $\lambda \neq \lambda'$}.
\end{equation*}
It follows that $\mathcal{B}_{k,\chi}$ is linearly independent. Now, we compute  
\begin{equation*}
\begin{aligned}
& \sum_{\substack{\chi,\\ \chi(-1) = (-1)^k}} \lvert \mathcal{O}_{N,\chi}^{(0)} \rvert \\
& = \sum_{\lambda \in \mathcal{O}_{N,\chi}^{(0)}} \bigl\lvert \{\chi \mid \lambda \in \mathcal{O}_{N,\chi}^{(0)}, \chi(-1) = (-1)^k \}\bigr\rvert \\
& = \begin{cases}
\frac{1}{2} \sum_{\delta \mid N} \varphi(\delta) \varphi(\frac{N}{\delta}), \hspace{6cm}\substack{N \neq 4}\\
3, \hspace{8cm}\substack{\text{$N=4$ and $k$ even}}\\
2,  \hspace{8cm}\substack{\text{$N=4$ and $k$ odd}} 
\end{cases}\\
& = \text{dim}_{\mathbb{C}} \mathcal{E}_{k}^{\ast}\big(\Gamma_{1}(N), \mathbb{C}\big)
\end{aligned}
\end{equation*}
where $\varphi$ is Euler's totient function and the third line uses Lemma~\ref{lemma7.1new}(i) to count the number of good characters. Therefore, for $k \geq 3$, the collection $\mathcal{B}_{k,\chi}$ is already a maximal linearly independent subset of $\mathcal{E}_{k}(N,\chi)$. Now suppose $k = 2$. Then 
\[\lvert \mathcal{O}_{N,\text{trivial}}^{(0)} \rvert = \mathcal{E}^{\ast}_2\big(\Gamma_{0}(N), \mathbb{C}\big) = 1 + \text{dim}_{\mathbb{C}} \mathcal{E}_2\big(\Gamma_{0}(N)\big)\] and $\text{dim}_{\mathbb{C}} \mathcal{E}_2\big(\Gamma_{1}(N)\big) = \mathcal{E}^{\ast}_2\big(\Gamma_{1}(N), \mathbb{C}\big) -1$. It follows that in this case also $\mathcal{B}_{k,\chi}$ is a maximal linearly independent subset of $\mathcal{E}_{k}(N,\chi)$ whenever $\chi$ is nontrivial. 
\end{proof}

One can also write the sums in \eqref{7.1new} and \eqref{7.2new} using the unnormalized series instead of the spectral series. Lemma~\ref{lemma7.1new} goes through without any change in the unnormalized setting. Moreover, the isomorphism $\Psi_{\Gamma_1(N), \mathcal{A}}$ in Section~\ref{section5.1} along with the theorem above demonstrates that the unnormalized series with nebentypus indeed yields a basis for $\mathcal{E}_k(N, \chi)$.

\subsection{The twisted Eichler-Shimura theory}
\label{section7.2new}
Let $N$ be a positive integer and $\chi \in \widehat{U}_N$. Then one can twist the $\Gamma_0(N)$-module structure on $V_{k-2,\mathbb{C}}$ to obtain another $\Gamma_{0}(N)$-module $V_{k-2,\mathbb{C}}^{\chi}$ whose underlying space is $V_{k-2,\mathbb{C}}$ and the $\Gamma_{0}(N)$-action is $P(X)\lvert_{\chi, \gamma} = \chi(\gamma)^{-1} P(X)\lvert_{\gamma}$. In this setting, one can associate a twisted cocycle 
\[\mathsf{p}_{k, \chi}^{\tau_0}(f) : \Gamma_0(N) \to V_{k-2, \mathbb{C}}^{\chi}, \hspace{.3cm} \gamma \mapsto  \int^{\tau_0}_{\gamma^{-1}\tau_0}f(\tau)(X-\tau)^{k-2}d\tau\]
which gives rise to a twisted Eichler-Shimura map 
\[[\text{ES}_k]_{\chi}: \mathcal{M}_{k}(N, \chi) \to H^{1}\big(\Gamma_0(N), V^{\chi}_{k-2,\mathbb{C}}\big), \hspace{.3cm} f \mapsto [\mathsf{p}_{k, \chi}^{\tau_0}(f)].\]
As before the image of $\mathsf{p}_{k, \chi}^{\tau_0}(f)$ in cohomology is independent of the choice of the base point $\tau_0$. We also need the space of antiholomorphic cusp forms 
\[\mathcal{S}_{k}^{a}(N,\chi) := \{\bar{f} \mid f \in \mathcal{S}_{k}(\Gamma_{1}(N)); \hspace{.1cm} \text{$\overline{f\lvert_{\gamma}} = \chi(\gamma)\bar{f}$ for each  $\gamma \in \Gamma_{0}(N)$}\}.\]
Note that $\mathcal{S}_{k}^{a}(N,\chi) = \overline{\mathcal{S}_{k}(N,\bar{\chi})}$. One extends the Eichler-Shimura homomorphism to this space as follows: 
\[[\text{ES}_k]_{\chi}^{a}: \mathcal{S}_{k}^{a}(N,\chi) \to H^{1}\big(\Gamma_0(N), V^{\chi}_{k-2,\mathbb{C}}\big), \hspace{.3cm} f \mapsto \#[\text{ES}_k]_{\bar{\chi}}(\bar{f}).\]
Here $\#$ is complex conjugation on the cohomology arising from the conjugation on the coefficients. The twisted Eichler-Shimura homomorphisms together give rise to a $\mathbb{C}$-linear homormorphism
\begin{equation}
\label{7.5new}
[\text{ES}_k]_{\chi} \oplus [\text{ES}_k]_{\chi}^{a}: \mathcal{M}_{k}(N, \chi) \oplus \mathcal{S}^{a}_{k}(N, \chi) \to H^{1}\big(\Gamma_{0}(N), V_{k-2,\mathbb{C}}^{\chi}\big)
\end{equation}
that is compatible with the twisted action of the Hecke algebra. As before, the Eichler-Shimura isomorphism for the cusp forms already yields a description for the space of cusp forms:  
\begin{equation*}
\mathcal{S}_{k}(N,\chi) \oplus \mathcal{S}^{a}_{k}(N,\chi) \xrightarrow{\cong} H^{1}_{c}\big(\Gamma_{0}(N), V_{k-2,\mathbb{C}}^{\chi}\big)
\end{equation*}
where $H^{1}_{c}\big(\Gamma_{0}(N), V_{k-2,\mathbb{C}}^{\chi}\big)$ refers to the parabolic subspace of cohomology. Hida \cite[6.3]{hida} used Hecke's Eisenstein series to show that \eqref{7.5new} is an isomorphism whenever $\chi$ is a primitive character. One can utilize our basis for the space of Eisenstein series parameterized by the cusps of $\Gamma_{0}(N)$ to complete his argument for all nebentypus characters. We, nevertheless, deduce the statement from the Eichler-Shimura isomorphism proved in the previous section. 

\begin{theorem}
\label{theorem7.4}
The map \eqref{7.5new} is an isomorphism. 
\end{theorem} 

\begin{proof}
We start with the untwisted Eichler-Shimura isomorphism  
\begin{equation}
	\label{7.6new}
	\mathcal{M}_k\big(\Gamma_{1}(N)\big) \oplus \mathcal{S}^{a}_k\big(\Gamma_{1}(N)\big) \xrightarrow{\cong} H^{1}\big(\Gamma_{1}(N), V_{k-2,\mathbb{C}}\big).
\end{equation}
One considers the $\chi$-eigenspaces for the action of diamond operators to discover that the isomorphism \eqref{7.6new} maps $\mathcal{M}_k(N,\chi) \oplus \mathcal{S}^{a}_k(N, \chi)$ isomorphically onto $H^{1}\big(\Gamma_{1}(N), V_{k-2,\mathbb{C}}\big)[\chi]$. Now the Eichler-Shimura homomorphism and its twisted avatar together yield a commutative diagram
\begin{equation*}
	\begin{tikzcd}
		& H^{1}\big(\Gamma_{0}(N), V_{k-2,\mathbb{C}}^{\chi}\big) \arrow[d]\\
		\mathcal{M}_k(N,\chi) \oplus \mathcal{S}^{a}_k(N, \chi) \arrow[r, "\text{\eqref{7.6new}}"] \arrow[ur, "\text{\eqref{7.5new}}"]& H^{1}\big(\Gamma_{1}(N), V_{k-2,\mathbb{C}}\big)
	\end{tikzcd}
\end{equation*}
where the vertical arrow is the restriction map. Note that the restriction of $\Gamma_{0}(N)$ structure on $V_{k-2, \mathbb{C}}^{\chi}$ to $\Gamma_{1}(N)$ is the same as $\Gamma_{1}(N)$-module $V_{k-2, \mathbb{C}}$. Since $\Gamma_{1}(N)$ is a normal subgroup of $\Gamma_{0}(N)$ the vertical arrow is injective, and surjective onto the fixed subspace $H^{1}\big(\Gamma_{1}(N), V_{k-2,\mathbb{C}}^{\chi}\big)^{\Gamma_{0}(N)}$. But the latter space equals $H^{1}\big(\Gamma_{1}(N), V_{k-2,\mathbb{C}}\big)[\chi]$. Therefore, the up-right arrow in the diagram above must be an isomorphism.  
\end{proof}

\subsection{Hecke action on the spectral basis}
\label{section7.3new}
This subsection aims to study the action of the Hecke algebras \cite[5.2]{ds}
\begin{equation*}
	\begin{aligned}
		\mathbb{T}_{\Gamma_{1}(N)}= \mathbb{Z}[\langle d \rangle, T_{p} \mid d \in (\mathbb{Z}/N\mathbb{Z})^{\times}, p \nmid N], \hspace{.2cm} \mathbb{T}_{\Gamma_{0}(N)} = \mathbb{Z}[T_{p} \mid p \nmid N]
	\end{aligned}
\end{equation*}
on the spectral basis of $\Gamma_{1}(N)$ and $\Gamma_{0}(N)$. As usual, here $\langle \cdot \rangle$ denotes the diamond operator arising from the action of $\Gamma_{0}(N)$. Recall that the Hecke eigenvalues of $\mathbb{T}_{\Gamma_{1}(N)}$ on the space of Eisenstein series are given by the linear combinations of cyclotomic characters modulo $N$. In other words, the Hecke action on the eigenbasis is defined only over the ring of cyclotomic integers. Our computations show that the Hecke action on the spectral basis is defined over $\mathbb{Z}$, illustrating a convenient aspect of the spectral basis. As before, we freely make use of the spectral bases listed in Appendix~\ref{appendix}.  

\begin{lemma}
\label{lemma7.4new}
Let $\bar{\lambda} = \overline{(\lambda_1, \lambda_2)} \in \Lambda_N$. Then 
\begin{enumerate}[label=(\roman*), align=left, leftmargin=0pt]
\item $\langle d \rangle \big(E_{k}^{(1)}(\bar{\lambda}, N)\big) = E_{k}^{(1)}\big((\bar{d}^{-1}\bar{\lambda}_1, \bar{d}\bar{\lambda}_2), N\big)$; \hspace{.3cm}$(\substack{\gcd(d, N) = 1})$
\item $T_p\big(E_{k}^{(1)}(\bar{\lambda}, N)\big) = p^{k-1} E_{k}^{(1)}\big((\bar{\lambda}_1, \bar{p}\bar{\lambda}_2), N\big) + E_{k}^{(1)}\big((\bar{p}^{-1}\bar{\lambda}_1, \bar{\lambda}_2), N\big)$.\hspace{.3cm}$(\substack{p \nmid N})$   
\end{enumerate}
\end{lemma}

\begin{proof}
\begin{enumerate}[label=(\roman*), align=left, leftmargin=0pt]
\item Let $\gamma = \begin{pmatrix}
			a & b \\
			c & d'
		\end{pmatrix} \in \Gamma_0(N)$ be a matrix with $d' \equiv d($mod $N)$. Then
		\[\langle d \rangle \big(E_{k}^{(1)}(\bar{\lambda}, N)\big) = E_{k}^{(1)}(\bar{\lambda}, N)\lvert_{\gamma} = E_{k}^{(1)}\big((\bar{d}^{-1}\bar{\lambda}_1, \bar{d}\bar{\lambda}_2), N\big).\] 
		
\item In light of \eqref{5.3} it suffices to verify the identity with the $E$-series replaced by the $G$-series. The formula concerning the action of $T_p$ on Fourier series \cite[p.171]{ds} demonstrates that
\begin{equation*} 
\begin{aligned} 
T_p\big(G_{k}^{(1)}(\bar{\lambda}, N)\big) & = \sum_{j \geq 0}\big(A_{k, jp}^{(1)}(\bar{\lambda}) + p^{k-1}A_{k, j/p}^{(1)}(\bar{p}^{-1}\bar{\lambda}_1, \bar{p}\bar{\lambda}_2)\big)\mathbf{e}(j\tau) \\
& \hspace{6cm}+ (p^{k-1}+1)P_{k}^{(1)}(\tau)
\end{aligned} 
\end{equation*} 
where $A_{k, j/p}^{(1)}$ is zero unless $p \mid j$. The identity for the unnormalized series now follows from the formula in Lemma~\ref{lemmaA.2}.   
\end{enumerate}
\end{proof}

We use the formulas for $\Gamma_{1}(N)$ to compute the Hecke action for $\Gamma_{0}(N)$.  

\begin{corollary}
\label{corollary7.5new}
Suppose that $p$ is a prime with $p \nmid N$. We have
\begin{equation*}
\begin{aligned}
T_{p}\big(E_{k}^{(0)}((0, 1), N)\big) & = (p^{k-1}+1)E_{k}^{(0)}((0, 1), N),\\
T_{p}\big(E_{k}^{(0)}((\frac{N}{2}, 1), N)\big) & = (p^{k-1}+1)E_{k}^{(0)}\big((\frac{N}{2}, 1), N\big),\\
T_{p}\big(E_{k}^{(0)}((\delta, \lambda_0), N)\big) & = p^{k-1}E_{k}^{(0)}\big((\delta, p\lambda_0), N\big) + E_{k}^{(0)}\big((\delta, p^{-1}\lambda_0), N\big).   
\end{aligned}
\end{equation*}
Here $p^{-1}\lambda_0$ refers to an integer that equals $p^{-1}\lambda_0$ modulo $N$.  
\end{corollary} 
\begin{proof}
Follows from \eqref{A.2}, \eqref{A.3}, and Lemma~\ref{lemma7.4new}. 
\end{proof}

The identities in Lemma~\ref{lemma7.4new} and Corollary~\ref{corollary7.5new} show that the action of Hecke operators on the orbital sums is defined over $\mathbb{Z}$. If $\mathcal{G} = \Gamma_{0}(N)$ with $N \geq 3$, then all the spectral Eisenstein series differ from the orbital sums by a factor of $\frac{1}{2}$. In other words, the action of Hecke operators on the spectral basis is also defined over $\mathbb{Z}$. Now suppose $\mathcal{G} = \Gamma_{1}(N)$. Then the factor $\frac{1}{2}$ appears only if $N = 4$, $\lambda \equiv (2,1) (\text{mod } 4)$, and $k$ is even. Note that if one term in an identity in Lemma~\ref{lemma7.4new} appears with the factor, then all terms in it require the same scaling. As a consequence, we have the following result:

\begin{proposition}
\label{proposition7.6new}
Let $\mathbbm{k}$ be a subring of $\mathbb{C}$. Then $\mathcal{E}_{k}(\mathcal{G}, \mathbbm{k})$ is stable under the action of $\mathbb{T}_{\mathcal{G}} \otimes \mathbbm{k}$ on $\mathcal{E}_{k}(\mathcal{G})$ where $\mathcal{G}$ is either $\Gamma_{1}(N)$ or $\Gamma_{0}(N)$.   
\end{proposition}

\begin{remark}
\label{remark7.7new}
The proof of Lemma~\ref{lemma7.4new} shows that the results of this subsection also hold for the bases arising from the unnormalized series. 
\end{remark}

\subsection{Conclusion: Further directions}
\label{section7.4new}
The primary intention behind our explicit study of the spectral Eisenstein series on general congruence subgroups is to discover new results of interest, generalizing the traditional framework of the arithmetic theory for $\Gamma_{1}(N)$. In a related work, the author \cite{sahu} has shown that the image of the spectral basis on a general congruence subgroup containing $\Gamma(N)$, under the Eichler-Shimura homomorphism, is defined over $\mathbb{Q}(\zeta_N)$. This result automatically implies that the action of a double coset operator on $\mathcal{E}_{k}(\mathcal{G})$ preserves $\mathcal{E}_{k}(\mathcal{G}, \mathbbm{k})$ where $\mathbbm{k}$ is a subfield of $\mathbb{C}$ containing $\mathbb{Q}(\zeta_N)$; cf. Proposition~\ref{proposition7.6new}. One of the possible directions from here is to write down the Hecke eigenbasis for the space of Eisenstein series on general congruence subgroups. Note that the family of congruence subgroups listed in \eqref{3.2new} already gives a large class of test cases to implement this program, which may extend to a general congruence subgroup by a Larcher-type group theoretic argument. Explicit knowledge of cusps and Hecke eigenbasis relates to subtle topics, such as the Eisenstein ideal and modular symbols; cf. \cite{banerjee-merel}. The author aspires to investigate some of these topics on more general congruence subgroups in the future.

\section*{Acknowledgments}
I am indebted to Prof. Eknath Ghate for stimulating conversations and generous encouragement. I also wish to thank the School of Mathematics, TIFR, for providing an exciting environment to study and work.

\appendix

\section{Explicit bases}
\label{appendix}
The appendix presents explicit formulas for the spectral basis attached to familiar families of congruence subgroups using the theory in Section~\ref{section4}. We begin by fixing a suitable choice of representatives $\mathcal{A}_N$ for the $\{\pm \text{Id}\}$-action on $\Lambda_N$ that helps us to separate an equivalent pair of regular orbits. For convenience one represents elements of $\Lambda_N$ by their lifts inside $[0, N) \times [0, N)$. Set $\mathcal{A}_{N} = \Lambda_{N}$ if $N \in \{1,2\}$. Assume that $N \geq 3$. Write 
\begin{gather*}
	\mathcal{A}_{N} = \mathcal{A}_{N}^{I} \amalg \mathcal{A}_{N}^{II} \amalg \mathcal{A}_{N}^{III};\\
	\mathcal{A}_{N}^{I} = \{(0, \lambda_2) \mid \substack{1 \leq \lambda_2 < \frac{N}{2},\; \gcd(\lambda_2, N)=1}\},\\
	\mathcal{A}_{N}^{II} = \{(\lambda_1, \lambda_2) \mid \substack{1 \leq \lambda_1 < \frac{N}{2}, \; 0 \leq \lambda_2 \leq N-1,\; \gcd(\lambda_1, \lambda_2, N)=1}\},\\
	\mathcal{A}_{N}^{III} = \begin{cases}
		\{(\frac{N}{2}, \lambda_2), (\frac{N}{2}, \frac{N}{2} + \lambda_2) \mid \substack{1 \leq \lambda_2 < \frac{N}{4},\, \gcd(\lambda_2, \frac{N}{2}) = 1}\}, & \substack{\text{$N$ even and $> 4$;}}\\
		\{(2, 1)\}, & \substack{\text{$N = 4$;}}\\
		\emptyset, & \substack{\text{$N$ odd.}} 	
	\end{cases}
\end{gather*}
In the following discussion, if $\frac{N}{2}$ occurs as part of a list, then it is implicit that the corresponding member appears only if $N$ is even.

\subsection{Example: $\Gamma(N,t)$}
\label{appendixA.1}
Let $t$ be a positive integer so that $t \mid N$. With $(N,t)$ one associates a congruence subgroup described by 
\[\Gamma(N,t) := \Big\{\begin{pmatrix}
	a & b\\
	c & d
\end{pmatrix}\, \bigl\lvert\, \substack{a \equiv d \equiv 1(\text{mod }N), \\ c \equiv 0(\text{mod }N),\, b \equiv 0(\text{mod }t)}\Big\}.\]
Note that $\Gamma(N, 1) = \Gamma_{1}(N)$ and $\Gamma(N,N) = \Gamma(N)$. Thus the family $\Gamma(N,t)$ provides an interpolation between the two well-studied congruence subgroups. Let $(\lambda_1, \lambda_2), (\lambda'_1, \lambda'_2) \in \Lambda_N$. Then \[(\lambda_1, \lambda_2) \Gamma(N,t) = (\lambda'_1, \lambda'_2) \Gamma(N,t)\] if and only if $(\lambda_1', \lambda_2') \equiv  (\lambda_1, \lambda_2 + j t\lambda_1)($mod $N)$ for some $j \in \mathbb{Z}$. This equivalence allow us to locate the orbits in $\Lambda_{N}$. 

\paragraph{Orbits for $\Gamma(N,t)$.} Write $\mathcal{O}_{(\lambda_1, \lambda_2),N}^{(t)} = (\lambda_1, \lambda_2)\Gamma(N,t) \subseteq \Lambda_{N}$.
\begin{enumerate}[label=\textbullet, align=left, leftmargin=10pt]
\item $(N=2)$ The set $\mathcal{A}_2$ consists of two $\Gamma_1(2)$ orbits, namely, \[\text{$\mathcal{O}^{(1)}_{(0,1),2} = \{(0,1)\}$ and $\mathcal{O}^{(1)}_{(1,0),2} = \{(1,0), (1,1)\}$.}\]
If $t = 2$ then each point of $\mathcal{A}_2$ determine a singleton orbit for $\Gamma(2)$. 
	
\item $(N \geq 3)$ The elements of type I are fixed under the action of $\Gamma(N, t)$ and give rise to singleton orbits $\mathcal{O}^{(t)}_{(0, \lambda_2), N} = \{(0, \lambda_2)\}$ for each choice of $\lambda_2$.
	
\item $(N \geq 3)$ The elements of type II define orbits of the form \[\mathcal{O}^{(t)}_{(\lambda_1, \lambda_0),N} = \{(\lambda_1, \lambda_0 + j \gcd(t\lambda_1, N)) \mid \substack{0 \leq j \leq \frac{N}{\gcd (t\lambda_1, N)} - 1}\}\] where $\lambda_1$ is as in the definition of $\mathcal{A}_{N}^{II}$ and $0 \leq \lambda_0 \leq \gcd(t\lambda_1, N) - 1$ with $\gcd (\lambda_0, t\lambda_1, N) = 1$. 
	
\item ($N$ even and $>4$) If $t$ is odd then the elements of type III lie in orbits of the form  \[\mathcal{O}^{(t)}_{(\frac{N}{2}, \lambda_2), N} = \{(\frac{N}{2}, \lambda_2), (\frac{N}{2}, \frac{N}{2} + \lambda_2)\}.\]
where $1 \leq \lambda_2 < \frac{N}{4}$ and $\gcd(\lambda_2, \frac{N}{2}) = 1$. Now suppose $t$ is even. Then $\Gamma(N,t)$ acts trivially on elements of type III and each orbit is singleton.
	
\item ($N=4$) If $t =1$, then the orbit of $(2, 1)$ is $\{(2,1), (2,3)\}$. But if $t$ is even then the orbit of $(2, 1)$ equals $\{(2,1)\}$. 
\end{enumerate}

\begin{remark}
\label{remarkA.1}
Let $N \geq 3$. The list above shows that the $\Gamma(N,t)$-orbits of the elements, except $\mathcal{O}^{(1)}_{(2,1),4}$, are regular and contained in $\mathcal{A}_{N}$. The orbit $\mathcal{O}^{(1)}_{(2,1),4}$ is irregular and corresponds to the irregular cusp $\frac{1}{2}$ of $\Gamma_{1}(4)$. In particular $\mathcal{A}_{N}$ is $\Gamma(N,t)$-admissible in the sense of Section~\ref{section3}.	
\end{remark}

Let $\lambda = (\lambda_1, \lambda_2) \in \Lambda_N$. Write \[\delta_t(\lambda) := \gcd(t\lambda_1, N).\] We abbreviate $\delta_{t}(\lambda)$ as $\delta_{t}$ if $\lambda$ is clear from the context. With $\lambda$ one associates an orbital sum on $\Gamma(N,t)$ described as 
\begin{equation}
\label{A.1}
E_{k}^{(t)}(\lambda, N) := \sum_{n=0}^{\frac{N}{\delta_t} - 1}E_{k}(\overline{(\lambda_1,\lambda_2 + \delta_t n)}, N).
\end{equation} 
Set  
\[\mathcal{O}_{k,N}^{(t)} := \begin{cases}
	\emptyset, & \substack{\text{$k$ odd and $N \in \{1,2\}$;}}\\
	\{(0,0)\}, & \substack{\text{$k$ even and $N = t =1$;}}\\
	\{(0,1), (1,0)\}, & \substack{\text{$k$ even and $N = 2$, $t=1$;}}\\
	\{(0,1), (1,0), (1,1)\}, & \substack{\text{$k$ even and $N = t= 2$;}}\\
	\{(0, \lambda_2)\} \cup \{(\lambda_1, \lambda_0)\} \cup \{(\frac{N}{2}, \lambda_2)\}, & \substack{\text{$N \geq 3$, $N \neq 4$, and $t$ odd;}}\\
	\{(0, \lambda_2)\} \cup \{(\lambda_1, \lambda_0)\} \cup \{(\frac{N}{2}, \lambda_2), (\frac{N}{2}, \frac{N}{2} + \lambda_2)\}, & \substack{\text{$N \geq 3$, $N \neq 4$, and $t$ even;}}\\
	\{(0, \lambda_2)\} \cup \{(\lambda_1, \lambda_0)\}, & \substack{\text{$k$ odd and $(N,t) = (4,1)$;}}\\
	\{(0, \lambda_2)\} \cup \{(\lambda_1, \lambda_0)\} \cup \{(2, 1)\}, & \substack{\text{$k$ even and $(N,t) = (4,1)$;}} 
\end{cases}\]
where the ranges for the parameters are as in the list above. For simplicity, assume that $N \geq 3$ if $k$ is odd, and write $\mathcal{O}_{\lambda}$ for the orbit of $\lambda$. With each $\lambda \in \mathcal{O}_{k,N}^{(t)}$ one can attach a cusp $x_{\lambda} := x_{\mathcal{O}_\lambda}$ as in the discussion around \eqref{4.6new}. Then \[\{x_{\lambda} \mid \lambda \in \mathcal{O}_{k,N}^{(t)}\}\] is the collection of all cusps, resp. regular cusps, if $k$ is even, resp. odd, without repetition. Moreover, in the light of Remark~\ref{remarkA.1} we have 
\[E_{k, x_{\lambda}} = \begin{cases}
	\frac{1}{2} E_k^{(t)}(\lambda, N), & \text{if $(N,t) = (4,1)$, $k$ even, $\lambda = (2,1)$};\\
	E_k^{(t)}(\lambda, N), & \text{otherwise.} 
\end{cases}\] 
Here the cusp $x_{(0,1)} = \infty$ has the smallest cusp amplitude, and the size of its $\Gamma(N,t)$-orbit in $\mathcal{C}\big(\Gamma(N)\big)$ equals $1$. Now one can write down a spectral basis for $\Gamma(N,t)$ using Theorem~\ref{theorem4.5new}. To perform explicit calculations we also need the unnormalized version of these series. Set  
\begin{equation*}
G_{k}^{(t)}(\lambda, N) := \sum_{n=0}^{\frac{N}{\delta_t} - 1}G_{k}(\overline{(\lambda_1, \lambda_2 +\delta_t n)}, N).
\end{equation*} 

\begin{lemma}
\label{lemmaA.2}
The Fourier expansion formula the unnormalized orbital sum is as follows:  
\begin{gather*}
G_{k}^{(t)}(\lambda, N) = \sum_{j \geq 0}A_{k,j}^{(t)}(\lambda)\mathbf{e}(\frac{j \delta_1 \tau} {\delta_{t}}) + P_{k}^{(t)}(\tau);\\
A_{k,0}^{(t)}(\lambda) := \mathbbm{1}_{\mathbb{Z}/N\mathbb{Z}}(\lambda_1, 0) \sum_{\substack{n \equiv \lambda_2(\delta_t),\\ n \neq 0}} \frac{1}{n^k},\;P_k^{(t)}(\tau) := - \mathbbm{1}_{\mathbb{Z}}(k,2)\frac{\pi}{\delta_t N\emph{Im}(\tau)},\\
A_{k,j}^{(t)}(\bar{\lambda}) := \frac{(-2\pi i)^k}{(k-1)! \delta_t^k} \sum_{\substack{r \in \mathbb{Z} - \{0\},\\ r \mid j, \frac{j}{r}\equiv \frac{\lambda_1} {\delta_1}(\frac{N} {\delta_1})}} \emph{sgn}(r)r^{k-1}\mathbf{e}(\frac{r\lambda_2}{\delta_t})\hspace{.3cm}(\substack{j \geq 1})
\end{gather*}
where $\delta_{1} = \gcd (\lambda_1, N)$ and $\delta_{t} = \gcd (\lambda_1t, N)$.
\end{lemma}

\begin{proof}
The identity is a direct consequence of the Fourier expansion formula \eqref{2.6new} for the unnormalized series on $\Gamma(N)$. 
\end{proof}

\subsection{Example: $\Gamma_{0}(N)$}
\label{appendixA.2}
If $N \in \{1,2\}$ then $\Gamma_{0}(N) = \Gamma_{1}(N)$ and there is nothing new. Suppose that $N$ is a positive integer $\geq 3$. Let $(\lambda_1, \lambda_2), (\lambda'_1, \lambda'_2) \in \Lambda_N$. A simple analysis shows that $(\lambda_1, \lambda_2) \Gamma_0(N) = (\lambda'_1, \lambda'_2) \Gamma_0(N)$ if and only if \[(\lambda_1', a\lambda_2') \equiv (a\lambda_1, \lambda_2 + j \lambda_1)(\text{mod }N)\] for some $a, j \in \mathbb{Z}$ with $\gcd(a, N) = 1$. 

\paragraph{Orbits for $\Gamma_0(N)$.} \cite[p.103]{ds} Write $\mathcal{O}_{(\lambda_1, \lambda_2),N}^{(0)} = (\lambda_1, \lambda_2) \Gamma_{0}(N) \subseteq \Lambda_{N}$. 

\begin{enumerate}[label=\textbullet, align=left, leftmargin=10pt]
\item The elements of type I lie in a single orbit $\mathcal{O}^{(0)}_{(0,1), N} = \mathcal{A}^{I}_{N} \amalg -\mathcal{A}^{I}_{N}$.
\item The elements of type II give rise to orbits parameterized by $(\delta, \lambda_0)$ where $\delta$ is divisor of $N$ satisfying $1 \leq \delta < \frac{N}{2}$ and $0 \leq \lambda_0 \leq \gcd(\delta, \frac{N}{\delta}) -1$ is an integer such that $\gcd (\lambda_0, \delta, \frac{N}{\delta}) = 1$. Set $\mathcal{X}_{(\delta, \lambda_0)} = \amalg \mathcal{O}^{(1)}_{(a\delta, \lambda_2)}$ where $1 \leq a < \frac{N}{2\delta}$ with $\gcd(a, \frac{N}{\delta}) = 1$ and $0 \leq \lambda_2 \leq \delta - 1$ so that $\gcd(\lambda_2, \delta) = 1$ and $a \lambda_2 \equiv \lambda_0\big($mod $\gcd(\delta, \frac{N}{\delta})\big)$. Then \[\mathcal{O}^{(0)}_{(\delta, \lambda_0), N} = \mathcal{X}_{(\delta, \lambda_0)} \amalg - \mathcal{X}_{(\delta, \lambda_0)}.\] 
\item The elements of type III give rise to a single orbit $\mathcal{O}^{(0)}_{(\frac{N}{2}, 1), N} = \mathcal{A}_{N}^{III} \amalg -\mathcal{A}_{N}^{III}$.    
\end{enumerate}

We first introduce a few orbital sums that facilitate the description of the basis. Here $-\text{Id} \in \Gamma_{0}(N)$ and it suffices to consider the case that $k$ is even. Write 
\begin{equation}
\label{A.2}
\begin{aligned}
E_{k}^{(0)}\big((0, 1), N\big) &:= \sum_{\substack{0 \leq \lambda_2 \leq N-1,\\ \gcd(\lambda_2, N) = 1}} E_{k}^{(1)}\big((0, \lambda_2), N\big),\\
E_{k}^{(0)}\big((\frac{N}{2}, 1), N\big)  &:= \sum_{\substack{0 \leq \lambda_2 \leq \frac{N}{2}-1, \\ \gcd(\lambda_2, \frac{N}{2}) = 1}} E_{k}^{(1)}\big((\frac{N}{2}, \lambda_2), N\big).   
\end{aligned}
\end{equation}
Let $\delta$ be a divisor of $N$ with $1 \leq \delta < \frac{N}{2}$ and $\lambda_0$ be an integer coprime to $\gcd(\delta, \frac{N}{\delta})$. Set
\begin{equation}
\label{A.3}  
E_{k}^{(0)}\big((\delta, \lambda_0), N\big) := \sum_{(a, \lambda_2)} E_{k}^{(1)}\big((a\delta, \lambda_2), N\big)
\end{equation}
where the sum is over \[\Big\{(a, \lambda_2) \,\bigl\lvert\, \substack{0 \leq a \leq \frac{N}{\delta}-1, \gcd(a, \frac{N}{\delta}) = 1, 0 \leq \lambda_2 \leq \delta -1,\\ \gcd(\lambda_2, \delta)=1, a\lambda_2 \equiv \lambda_0 \big(\gcd(\delta, \frac{N}{\delta})\big)}\Big\}.\]   
The list of orbits yields a set of representatives \[\mathcal{O}_{N}^{(0)} := \big\{(0,1), (\frac{N}{2},1), (\delta, \lambda_0) \bigl\lvert\, \substack{\text{$\delta, \lambda_0$ as above and}\\0 \leq \lambda_0 \leq \gcd(\delta, \frac{N}{\delta})-1}\big\}.\] 
With each $\lambda \in \mathcal{O}_{N}^{(0)}$ one attaches a cusp $x_{\lambda}$ exactly as in Appendix~\ref{appendixA.1}. Note that $E_{k, x_{\lambda}} = \frac{1}{2}E_{k}^{(0)}(\lambda,N)$. The size of the orbit of $\lambda$ in $\Lambda_{N}$ equals 
\begin{equation} 
\label{A.4}
\lvert \lambda \Gamma_{0}(N) \rvert = \frac{N}{\delta_1} \frac{\varphi(\frac{N}{\delta_1})\varphi(\delta_1)}{\varphi\big(\gcd(\delta_1, \frac{N}{\delta_1})\big)} = \frac{N}{\delta_1} \varphi\big(\text{lcm}(\delta_1, \frac{N}{\delta_1})\big). 
\end{equation}
Here  $\varphi(\cdot)$ is the Euler's totient function, and $\delta_{1} = \gcd(\lambda_1,N)$.  In particular $\lvert (0,1) \Gamma_{0}(N) \rvert = \varphi(N)$. An elementary argument using \eqref{A.4} and the prime factorization of $N$ shows that $\frac{\lvert \lambda \Gamma_{0}(N) \rvert}{\varphi(N)}$ is an integer for each $\lambda \in \mathcal{O}^{(0)}_{N}$. Therefore, the cusp $x_{(0,1)} = \infty$ has the smallest cusp amplitude, and the size of its $\Gamma_{0}(N)$-orbit in $\mathcal{C}\big(\Gamma(N)\big)$ equals $\frac{\varphi(N)}{2}$. Now one can write down the spectral basis for $\Gamma_{0}(N)$ using Theorem~\ref{theorem4.5new}.

\end{document}